\documentclass[preprint,12pt]{elsarticle}

\journal{Nonlinear Analysis: Hybrid Systems}

\usepackage{mathtools, amssymb, amsthm}
\usepackage[hidelinks]{hyperref}
\usepackage{bbm}
\usepackage{pst-node}
\usepackage{tikz}
\usetikzlibrary{matrix}
\usepackage{dsfont}
\usepackage{booktabs}
\usepackage{multirow}
\usepackage{thmtools}
\usepackage{comment}
\usepackage{bm}
\usepackage{todonotes}

\newtheorem{theorem}{Theorem}
\newtheorem{lemma}[theorem]{Lemma}
\newtheorem{fact}[theorem]{Fact}

\newtheorem{example}{Example}
\newtheorem{corollary}[theorem]{Corollary}

\newcommand{\init}{\iota_{\mathit{init}}}

\newcommand{\alg}{\overline{\mathbb{Q}}}
\newcommand{\seq}[1]{(#1)_{n \in \mathbb{N}}}
\newcommand{\torus}{\mathbb{T}}
\newcommand{\nat}{\mathbb{N}}
\newcommand{\intg}{\mathbb{Z}}
\newcommand{\rel}{\mathbb{R}}
\newcommand{\rat}{\mathbb{Q}}
\newcommand{\com}{\mathbb{C}}
\newcommand{\ralg}{\rel \cap \alg}
\newcommand{\zerovec}{\mathbf{0}}

\newcommand{\Log}{\operatorname{Log}}

\newcommand{\MP}{\mathcal{MP}}
\newcommand{\M}{\mathcal{M}}

\newcommand{\Rea}{\operatorname{Re}}
\newcommand{\Ima}{\operatorname{Im}}

\newcommand{\st}{\colon}

\newcommand{\rebt}{\mathbb{R}_{\exp,\: \textrm{bt}}}
\newcommand{\rexp}{\mathbb{R}_{\exp}}

\newcommand{\Lcal}{\mathcal{L}}
\newcommand{\Tcal}{\mathcal{T}}

\newcommand{\im}{\mathbf{i}}
\newcommand{\Mb}{\mathbb{M}}
\newcommand{\Sb}{\mathbb{S}}

\begin{document}
	
	\begin{frontmatter}
		
		
		\title{Linear Dynamical Systems with Weight Functions}
		
		\author[tud]{Rajab Aghamov}
		\author[tud]{Christel Baier}
		\author[mpi]{Toghrul Karimov}
		\author[mpi]{Jo\"el Ouaknine}
		\author[tud,ul]{Jakob Piribauer}
		
		\address[mpi]{Max Planck Institute for Software Systems, Saarland Informatics Campus, Saarbrücken, Germany}
		\address[tud]{TU Dresden, Germany}
		\address[ul]{Leipzig University, Germany}
		
		\begin{abstract}
			In discrete-time linear dynamical systems (LDSs), a linear map is repeatedly applied to an initial vector yielding a sequence of vectors called the orbit of the system. A weight function assigning weights to the points in the orbit can be used to model quantitative aspects, such as resource consumption, of a system modelled by an LDS. This paper addresses the problems of how to compute the mean payoff, the total accumulated weight, and the discounted accumulated weight of the orbit
			 under  continuous weight functions as well as polynomial weight functions as a special case.
			 Additionally, weight functions that are definable in an o-minimal extension of the theory of the reals with exponentiation, which can be shown to be piecewise
			 continuous,  are considered. 
			   In particular, good ergodic properties of o-minimal weight functions, instrumental to the computation of the mean payoff, are established.
			   Besides general LDSs, the special cases of stochastic LDSs and LDSs with bounded orbits are addressed. 
			   Finally, the problem of deciding whether an energy constraint is satisfied by the weighted orbit, i.e., whether the accumulated weight never drops below a given bound, is analysed.
		\end{abstract}
		
		
		

	\end{frontmatter}
	
	\section*{Acknowledgements}
	\noindent
			Jo\"el Ouaknine is also part of Keble College, Oxford as emmy.network
Fellow. The authors were supported by DFG grant 389792660 as
part of TRR 248 (see https://perspicuous-computing.science) and
by BMBF (Federal Ministry of Education and Research) in DAAD
project 57616814 ((SECAI, School of Embedded and Composite
AI) as part of the program Konrad Zuse Schools of Excellence in
Artificial Intelligence.
\vspace{12pt}

\hrule

\section*{Related version}

\noindent
This paper is an extension of the conference paper \cite{HSCC}. New results on o-minimal weight functions (\autoref{sec:o-min-mean-payoff}) are included here.
Furthermore, the presentation of the technical parts has been improved and all  proofs omitted in the conference version are included.
\vspace{12pt}

\hrule

	
	\section{Introduction}

%
%

Dynamical systems, which describe how a system's state evolves over time, serve as a fundamental modeling framework across numerous disciplines.
A \emph{discrete-time linear dynamical system} (LDS) in an ambient space $\mathbb{R}^d$ begins at an initial point $q \in \mathbb{R}^d$. Its evolution follows a linear update rule governed by a matrix $M \in \mathbb{R}^{d \times d}$, which is applied iteratively to the system’s current state at each time step. This process generates the \emph{orbit} $(q, Mq, M^2q, \dots)$.
The study of LDSs is particularly significant because they represent one of the simplest classes of dynamical systems while still presenting numerous intricate challenges. Moreover, linearization plays a crucial role in control theory and engineering, where complex systems are frequently approximated using linear models (see, e.g., \cite{lee1987linearization,aranda1996linearization}), making it an essential tool for solving many real-world problems.

Algorithmic problems concerning LDSs form a lively area of research in computer science.
Reachability problems of LDSs are formidably hard: the Skolem Problem (hyperplane reachability) has been open for a hundred years, and the Positivity Problem (halfspace reachability) is at least as hard as certain long-standing open problems in Diophantine approximation whose solution would amount to a mathematical earthquake.
Decidability is known in low dimensions only, and is obtained through a combination of arguments from number theory and Diophantine approximation  \cite{tijdeman_distan_between_terms_algeb_recur_sequen,Verescchagin,ouaknine13_posit_probl_low_order_linear_recur_sequen,ouaknine_simple-positivity}.
However, it has recently been discovered that ``robust'' versions of many classical open problems of linear dynamical systems are decidable \cite{DCostaKMOSW22,karimov2024verification}. 
For example, for arbitrary $M \in \rat^{d\times d}$, $q \in \rat^d$, and semialgebraic $T$, it is decidable whether for every $\varepsilon > 0$ there exists $q'$ in the $\varepsilon$-ball around $q$ such that the orbit of $q'$ under $M$ reaches $T$.


%

%
The key new tool is o-minimality, which is a concept originating in logic and model theory, that is  defined by the following property: Any subset of $\rel^d$ that is definable over an o-minimal extension of the real field with exponentiation has finitely many connected components.
Although it can be seen as a simple property at first, the consequences of o-minimality are drastic: for example, o-minimality has recently been applied in a spectacular fashion to the the problem of counting rational points in a variety, which is the cornerstone problem of Diophantine geometry \cite{pila2014minimality}.

\begin{table*}[t]
	\caption{Overview of the results.}
	\label{tbl:overview}
    \vspace{12pt}
\hspace{-5pt}\scalebox{0.79}{
	\begin{tabular}{l |  l  l  l  l }
		\toprule
		& LDS type & weight function & algorithmic result \\
		\midrule
		mean payoff & arbitrary & polynomial & computable  & (Thm. \ref{thm:mean-payoff-polynomial-weight}) \\\cline{2-5}
		& bounded  & continuous  & integral representation  &(Thm. \ref{thm:mp_bounded}) \\
		&  orbit &  &   computable &\\\cline{2-5}
		& arbitrary &o-minimal & integral representation &(Thm.~\ref{thm:mean-payoff-omin-main}) \\
		&&&computable&\\\cline{2-5}
		&  stochastic,  & continuous & \multirow{3}{5cm}{computable with poly-many evaluations of the weight function} &(Thm. \ref{thm:mp_stochastic_irreducible}) \\
       &irreducible&& & \\
		&&& &\\\cline{2-5}
		& stochastic,  & continuous & 
		 \multirow{2}{5.5cm}{computable with exp-many evaluations of the weight function} 	
		 &(Thm. \ref{thm:mp_stochastic_reducible}) \\
        &reducible&& & \\
		 &&& &\\
		\midrule
		\multirow{2}{3cm}
        {total/discounted weight} & arbitrary & polynomial & computable &(Thm. \ref{thm:total_discounted}) \\
        \\
		\midrule
		 \multirow{2}{3cm}{satisfaction of energy constraints}& arbitrary & polynomial & decidable in dimension 3 &(Thm. \ref{thm:energy_dimension_3})  \\\cline{2-5}
		 &stochastic & linear & Positivity-hard &(Thm. \ref{thm:energy_Diophantine})  \\\cline{2-5}
		 & dimension 4 & polynomial & Diophantine-hard &(Thm. \ref{thm:diophantine4}) \\
		\bottomrule
	\end{tabular}
	}
\end{table*}

In this paper, we address  quantitative verification questions arising when dynamical systems are equipped with a weight function.
To the best of our knowledge, such quantitative verification tasks on weighted LDSs have not been investigated in the literature. 
We consider  \emph{continuous} weight functions  $w\colon \rel^d \to \rel$ assigning a weight to each state in the ambient space.
Such weight functions can be used to model various quantitative aspects of a system, such as
resource or energy consumption, rewards or utilities, or execution time for example.
 Given a weight function $w$, we obtain a sequence of weights of the states in the orbit $(w(q),w(Mq), w(M^2 q), \dots)$.
The goal of this paper is to provide algorithmic answers to the following typical questions arising for weighted systems:
\begin{itemize}
\item[a)] What is the \emph{mean payoff}, i.e., the average weight collected per step?
\item[b)] What is  the \emph{total accumulated weight} of the orbit and what is the so-called \emph{discounted  accumulated weight},
where weights obtained after $k$ time steps are discounted with a factor $\lambda^k$ for a given $\lambda\in(0,1)$?
\item[c)] Is there an $n\in \nat$ such that  the sum of weights obtained in the first $n$ steps lies below a given bound?
This problem is referred to as \emph{satisfaction of an energy constraint} because it  corresponds to determining whether a system ever runs out of energy
when weights model the energy used or gained per step.
\end{itemize}

 \begin{example}
A scheduler assigns tasks to $d$ different processors $P_1,\dots,P_d$ and that the load of the processors at different time steps 
can be modeled as an LDS with matrix $M\in\mathbb{Q}^{d\times d}$ and orbit $(M^k q)_{k\in \mathbb{N}}$ for a $q \in \mathbb{Q}^d$.
Further, assume for each processor $P_i$ there is an optimal load $\mu_i$ under which it works most efficiently.
To evaluate the scheduler, we want to know how closely the real loads in the long-run match the ideal loads. As a measure for how well a vector $x$ matches the vector $\mu$ of ideal loads, we use the average squared distance
\[
\delta_\mu(x) = \frac{1}{d} \sum_{i=1}^d (x_i - \mu_i)^2.
\]
To see how well the scheduler manages to get close to optimal loads in the long-run after a possible intialization phase, 
we consider the mean payoff of the orbit with respect to the weight function $\delta_\mu$, i.e.,
\[
\lim_{\ell \to \infty} \frac{1}{\ell} \sum_{k=0}^{\ell-1} \delta_\mu(M^k q).
\]
If, on the other hand, we know that the orbit will tend to the optimal loads for $k\to \infty$, we might instead also want to measure the total deviation
$\sum_{k=0}^\infty \delta_\mu(M^k q)$. If this value is small,  the orbit converges to the optimal loads rather quickly without large deviations initially.
\end{example}

In order to obtain algorithmic results, we consider different combinations of  classes of LDSs and  classes of  weight functions. Namely, besides arbitrary rational LDSs, we consider also LDSs with bounded orbit and \emph{stochastic LDSs}.
Stochastic LDSs occur in the context of the verification of probabilistic systems: For a finite-state Markov chain, the sequence of distributions over the state space
naturally forms an LDS. The initial distribution can be written as a vector $\init \in [0,1]^d$. Afterwards, the transition probability matrix $P$ can be repeatedly applied to obtain the distribution $P^k\init$ over states after $k$ steps. 
In contrast to the path semantics where a probability measure over infinite paths in a Markov chain is defined, the view of a Markov chain as an LDS is also called the \emph{distribution transformer semantics} of Markov chains (see, e.g., \cite{AgrawalApp}).

For the weight functions, on the one hand, we consider general continuous functions. Of course, for algorithmic results, we have to make additional assumptions on the computability or approximability of these functions. Furthermore, we consider the subclass of  polynomial weight functions with rational coefficients.
On the other hand, we consider weight functions that are definable in an o-minimal structure. These functions can be shown to be piecewise continuous, and include the vast family of weight functions that can be defined using arithmetic and exponentiation in the sense of first-order logic.
We show that the use of such weight functions results in a good limiting behavior of the weights of the orbit of an LDS allowing us to treat the mean payoff for arbitrary LDSs with such weight functions.

\paragraph*{Contribution}
Our contributions are as follows (see also Table \ref{tbl:overview}).
\begin{itemize}
\item[a)]
Mean payoff:
For rational LDSs equipped with a  polynomial weight function, we show that it is decidable whether the mean payoff exists, in which case it is rational and computable (\autoref{sec:polynomial_MP}).

We then show how to decide whether the orbit of a rational LDS is bounded.
If the orbit of a rational LDS is bounded, we show how to compute the set of accumulation points of the orbit and 
how to obtain a representation of the mean payoff as an integral using this set.
This integral can be approximated to arbitrary precision for any weight function $w$ that is sufficiently well-behaved (\autoref{sec:bounded_MP}).

Next, we consider LDSs with o-minimal weight functions. We show that the orbits of linear dynamical systems equipped with an o-minimal weight function--under some mild restrictions--is ergodic in the sense of time average being equal to space average (\autoref{sec:o-min-mean-payoff}).

Finally, we consider stochastic LDSs, which constitute a special case of LDSs with bounded orbits. 
Here, the orbit only has finitely many accumulation points.
We show that in case the transition matrix is irreducible,  one can compute polynomially many rational points in polynomial time such that the mean payoff 
is the arithmetic mean of the weight function evaluated at these points.
In the reducible case, on the other hand, exponentially many such rational points have to be computed (\autoref{sec:stochastic_MP}).

\item[b)]
Total and discounted accumulated weights:
For rational LDSs and polynomial weight functions, we prove that the total as well as the discounted accumulated weight of the orbit is  computable and rational if finite (\autoref{sec:total_reward}).

\item[c)]
Satisfaction of energy constraints:
First we prove that it is decidable whether an energy constraint is satisfied by an orbit under a polynomial weight function for LDSs of dimension $d=3$.
We furthermore provide two different hardness results regarding possible extensions of this decidability result:
At $d = 4$, the problem is hard with respect to certain open decision problems in Diophantine approximation that are at the moment wide open.
Further, also  restricting to stochastic LDSs and linear weight functions, does not lead to decidability in general:
we show that
the energy satisfaction problem is at least as hard as the Positivity Problem for linear recurrence sequences in this case.
The decidability status of the Positivity Problem is open, and it is known from \cite{ouaknine13_posit_probl_low_order_linear_recur_sequen} that its resolution would amount to major mathematical breakthroughs (sections \ref{sec:baker}--\ref{sec:hardness}).
 \end{itemize}

\paragraph*{Related work}
Verification problems for linear dynamical systems have been extensively studied for decades, starting with the question about the decidability of the Skolem \cite{tijdeman_distan_between_terms_algeb_recur_sequen,Verescchagin} and Positivity \cite{ouaknine13_posit_probl_low_order_linear_recur_sequen,ouaknine_simple-positivity} problems at low orders, which are special cases of the reachability problem for LDSs.
Decidable cases of the more general Model-Checking Problem for LDSs have been studied in \cite{AKK21,KLO22}.
In addition, decidability results for parametric LDSs \cite{Baier0JKLLOPW021} as well as various notions of robust verification \cite{BGV22,DCostaKMOSW22} have been obtained.
See \cite{KarimovKO022} for a survey of what is decidable about discrete-time linear dynamical systems. 

There is  very little related work on LDSs with weight functions. 
Closest to our work is the work by Kelmendi \cite{kelmendi2022}. There, it is shown that the \emph{natural density} (which is a notion of frequency) of visits of a rational LDS in a semialgebraic set always exists and is approximable to arbitrary precision.
A consequence of this result is that the mean payoff of a rational LDS with respect to a ``semialgebraic step function'', which 
takes a partition of the ambient space $\rel^d$ into finitely many semialgebraic sets $S_1,\dots,S_k$ and assigns a rational weight $w_i$ to the points in $S_i$, 
 can be approximated to arbitrary precision. This result is orthogonal to our results.

When it comes to Markov chains viewed as LDSs under the distribution transformer semantics, it is known that Skolem and Positivity-hardness results for general LDSs persist~\cite{AAOW2014}.
Vahanwala has recently shown \cite{mihir} that this is the case even for ergodic Markov chains.
In \cite{AgrawalApp}, Markov chains under the distribution transformer semantics are treated approximatively -- in contrast to our work -- by discretising  the probability value space $[0, 1]$ into a finite set of intervals and the problem to decide whether an approximation of the trajectory obtained in this way satisfies a property is studied.

Several connections between o-minimality and dynamical (or, more generally, cyber-physical) systems have  been established in the literature.
In \cite{lafferriere2000minimal}, Lafferriere et al.\ show that \emph{o-minimal hybrid systems} always admit a finite bisimulation.
In \cite{miller2011expansions}, Miller gives a dichotomy result concerning tameness of expansions of certain o-minimal structures with trajectories of linear dynamical systems.

	\section{Preliminaries}

	We write $\torus$ for $[0,1)$, $\im$ for the imaginary number, and $\zerovec$ for a vector of all zeros whose dimension will be clear from the context.
	We then have that $e^{\im2\pi \torus} = \{z \in \com \st |z| = 1\}$.
	We denote by $\{x\}$ the fractional part of $x \in \rel$, and by $e_i$ the $i$th standard basis vector of $\rel^d$, where $d$ will be clear from the context.

	\subsection{Linear dynamical systems}
	A (discrete-time) \emph{linear dynamical system} $(M,q)$ of dimension $d > 0$  consists of an update matrix $M\in \mathbb{R}^{d\times d}$ and an initial vector $q\in \mathbb{R}^d$.
	If the entries of $M$ and $q$ are rational, we say that the LDS is \emph{rational}.
	The \emph{orbit} of $(M,q)$ is the sequence $(M^k q)_{k\in\mathbb{N}}$.
	We say that the orbit of $(M,q)$ is bounded if it is bounded as a set under the Euclidean metric.
	An LDS is called \emph{stochastic} if the matrix $M$ and the initial vector $q$ have only non-negative entries and the entries of each column of $M$ as well as the entries of~$q$ sum up to $1$. 
	In this case we refer to the matrix $M$ as stochastic as well.\footnote{
	In order to keep the notation in line with the notation for general LDSs, we deviate from the standard convention that rows of stochastic matrices sum up to $1$ and that stochastic matrices are applied to  distributions by multiplication from the right.}

	\subsection{Algebraic numbers}
	A number $\alpha \in \com$ is \emph{algebraic} if there exists a polynomial $p \in \rat[x]$ such that $p(\alpha) = 0$. 
	Algebraic numbers form a subfield of ~$\com$ denoted by $\alg$. 
	The minimal polynomial of $\alpha \in \alg$ is the (unique) monic polynomial $p \in \rat[x]$ of the smallest degree such that $p(\alpha) = 0$. 
	The \emph{degree} of $\alpha$, denoted by $\deg(\alpha)$, is  the degree of the minimal polynomial of $\alpha$. 
	For each $\alpha \in \alg$ there exists a unique polynomial $P_\alpha = \sum_{i=0}^{d} a_i X^i \in  \intg[x]$ with $d = \deg(\alpha)$, called the \emph{defining polynomial} of $\alpha$, such that $P_{\alpha}(\alpha) = 0$ and $\gcd(a_0, \ldots, a_{d}) = 1$.
	The polynomial~$P_{\alpha}$ and the minimal polynomial of $\alpha$ have identical roots, and are \emph{square-free}, i.e., all of their roots appear with multiplicity one.
	The \emph{(naive) height} of~$\alpha$, denoted by $H(\alpha)$, is equal to $\max_{0\le i \le d} |a_i|$. 
	We represent an algebraic number $\alpha$ in computer memory by its defining polynomial $P_\alpha$ and sufficiently precise rational approximations of $\operatorname{Re}(\alpha), \operatorname{Im}(\alpha)$ to distinguish $\alpha$ from 
	other roots of $P_\alpha$.
	We denote by $||\alpha||$ the bit length of a representation of $\alpha \in \alg$.
	We can perform arithmetic effectively on algebraic numbers represented in this way \cite{cohen2013course}.

\subsection{Linear recurrence sequences}
\label{sec:lrs}
A sequence $\seq{u_n}$ is a \emph{linear recurrence sequence} 
over a ring $R \subseteq \com$ if there exists a positive integer $d$ and a \emph{recurrence relation} $(a_0,\ldots,a_{d-1}) \in R^d$ such that $u_{n+d} = \sum_{i=0}^{d-1}a_iu_{n+i}$ for all $n$. 
The \emph{order} of $\seq{u_n}$ is the smallest positive integer $d$ such that $\seq{u_n}$ satisfies a recurrence relation in $R^d$.
In this work, we will mostly encounter linear recurrence sequences over $\rat$, called \emph{rational} LRSs.
Examples of rational LRSs include the Fibonacci sequence, $u_n = p(n)$ for $p \in \rat[x]$, and $u_n = \cos (n\theta)$ where $\theta \in \{\arg(\lambda)\colon \lambda \in \rat(\im)\}$. 
We refer the reader to the books by Everest et al. \cite{everest-recurrence-sequences} and Kauers \& Paule~\cite{kauers-tetrahedron} for a detailed discussion of linear recurrence sequences.

An LRS $\seq{u_n}$ that is not eventually zero satisfies a unique minimal recurrence relation $u_{n+d} = \sum_{i=0}^{d-1}a_iu_{n+i}$ such that $d > 0$ and $a_0 \ne 0$.
Writing
$A = \begin{bmatrix}
	a_1 & \cdots & a_{d-1}
\end{bmatrix}$
and 
$q = \begin{bmatrix}
	u_{0} & \cdots & u_{d-1}
\end{bmatrix}^\top$, the matrix
\[
C \coloneqq
\begin{bmatrix}
	\zerovec & I_{d-1}\\
	a_0& A
\end{bmatrix} = 
\begin{bmatrix}
	0& 1&\cdots&0\\
	\vdots & \vdots & \ddots & \vdots \\
	0&0 & \cdots& 1\\
	a_0 & a_1 & \cdots & a_{d-1}
\end{bmatrix} \in R^{d \times d}
\]
is called the \emph{companion matrix} of $\seq{u_n}$.
We have that
\[
C^nq = (u_n,\ldots,u_{n-d+1})
\] 
and $u_n = e_1^\top C^n q$ for all $n \in \nat$, where $e_i$ denotes the $i$th standard basis vector.
As $a_0 \ne 0$, the matrix $C$ is invertible and does not have zero as an eigenvalue.

The \emph{characteristic polynomial} of $\seq{u_n}$ is $p(x) = x^d - \sum_{i=0}^{d-1} a_ix^i$.
Note that $p$ is identical to the characteristic polynomial $\det(xI-C)$ of the companion matrix $C$.
The \emph{eigenvalues} (also called the \emph{characteristic roots}) of $\seq{u_n}$ are the $d$ (possibly non-distinct) roots $\lambda_1, \ldots, \lambda_d$ of the characteristic polynomial $p$.
An LRS is 
\begin{itemize}
	\item \emph{simple} (or \emph{diagonalisable}) if its characteristic polynomial does not have a repeated root, and
	\item \emph{non-degenerate} if (i) all real eigenvalues are non-negative, and (ii) for every pair of distinct eigenvalues $\lambda_1, \lambda_2$, the ratio $\lambda_1/\lambda_2$ is not a root of unity.
\end{itemize}
For an LRS $\seq{u_n}$ there exists effectively computable $R$ such that for every $0 \le r  < R$, the sequence $u^{(r)}_n = u_{nR +r}$ is non-degenerate \cite[Section~1.1.9]{everest-recurrence-sequences}.
If $\seq{u_n}, \seq{v_n}$ are LRSs over a field $R$, and $\circ \in \{+, -, \cdot\}$, then $w_n = u_n \circ v_n$ also defines an LRS over $R$ \cite[Theorem~4.2]{kauers-tetrahedron}.
Moreover, if $\seq{u_n}$ and $\seq{v_n}$ are both simple, then so is $\seq{w_n}$.

\paragraph*{The exponential polynomial representation of an LRS}
Every LRS $\seq{u_n}$ of order $d > 0$ over $\alg$ can be written in the form
\begin{equation}
	\label{eq:exp-poly-form}
	u_n = \sum_{j=1}^m p_j(n) \lambda_j^n
\end{equation}
where (i) $m \ge 1$, $\lambda_1, \ldots, \lambda_m$ are the distinct non-zero characteristic roots of $\seq{u_n}$, and (ii) each $p_i$ is a non-zero polynomial with algebraic coefficients; see \cite[Chapter~1]{everest-recurrence-sequences}.
Whenever these conditions on $m, \lambda_i$ and $p_i$ are met, we say that the right-hand side is in the \emph{exponential polynomial form}.
Every LRS $\seq{u_n}$ that is not eventually zero has a unique representation of the form (\ref{eq:exp-poly-form}) where the right-hand side is in the exponential polynomial form.
Moreover, the right-hand side of (\ref{eq:exp-poly-form}) cannot be identically zero assuming (i-ii).
This is a folklore result, but we provide a proof for completeness.
\begin{lemma}
	\label{thm:exp-poly-of-id-zero}
	Let $u_n = \sum_{i=1}^m p_i(n)\lambda_i^n$ where $m\ge1$, $\lambda_1,\ldots,\lambda_m \in \alg$ are non-zero and pairwise distinct, and each $p_i \in \alg[x]$ is non-zero.
	Then, the sequence $\seq{u_n}$ is not identically zero.
	Specifically, there exists $0 \le n < d$, where $d = \sum_{i=1}^m (\deg(p_i) + 1)$, such that $u_n \ne 0$.
\end{lemma}
\begin{proof}
Suppose $\deg(p_k) \ge 1$ for some $1 \le k \le m$.
Consider the sequence $v_n = u_{n+1}-\lambda_ku_n$.
It will be of the form
\[
v_n = \sum_{i \in I} q_i(n)\lambda_i^n
\]
where $I \subseteq \{1, \ldots, m\}$ with $k \in I$, $\deg(q_k) < \deg(p_k)$, and for all $i \in I$, $q_i$ is not identically zero with $\deg(q_i) \le \deg(p_i)$.
Observe that if $\seq{u_n}$ is identically zero, then so is $\seq{v_n}$.
Moreover, if $v_n$ is non-zero, then either $u_n$ or $u_{n+1}$ is non-zero.
Repeating the process of constructing $v_n$ from $u_n$ at most $\sum_{i=1}^m \deg(p_i)$ times, we obtain 
\[
w_n = \sum_{i=1}^mc_i \lambda_i^n
\]
that is identically zero if $u_n$ is identically zero, where each $c_i$ is an algebraic number and at least one $c_i$ is non-zero.

It remains to argue that $w_n$ cannot be identically zero.
Consider the system of equations 
\[
\sum_{i=1}^{m} x_i\lambda_i^n = 0 \quad\quad \textrm{for $0 \le n < m$}.  
\]
We can write it as $M x = \zerovec$, where $x =  (x_1,\ldots,x_m)$ and $M$ is a Vandermonde matrix with $\det(M) = \prod_{i\ne j} (\lambda_i - \lambda_j)$.
Since $\lambda_1, \ldots, \lambda_m$ are distinct by assumption, $M$ is invertible and $Mx = 0$ if and only if $x = \zerovec$.
Since $c \ne \zerovec$, it follows that $w_n \ne 0$ for some $0 \le n < m$.
Hence there exists $n' \le n + \sum_{i=1}^m \deg(p_i) = n + (d - m) < d$ such that $u_{n'} \ne 0$.
\qedhere
\end{proof}
We can also characterise the exponential polynomial representations of real-valued LRS.
\begin{lemma}
	\label{thm:real-lrs-galois-condition}
	Let $\seq{u_n}$ be as in the statement of \autoref{thm:exp-poly-of-id-zero}.
	If $u_n \in \rel$ for all $n \in \nat$, then for every $1\le i \le m$ there exists $j$ with $1\leq j \leq m$ such that $p_j(n) = \overline{p_i}(n)$ and $\lambda_j = \overline{\lambda_i}$.
\end{lemma}
\begin{proof}
	We have $\overline{u_n} = \sum_{j=1}^m \overline{p_j}(n)\overline{\lambda_j}^n$.
	Moreover, $u_n = \overline{u_n}$ since $u_n \in \rel$ for all~$n$.
	The result then follows from the uniqueness of the exponential polynomial representation.
\end{proof}

Throughout this work we will encounter sequences of the form $u_n = p(M^nq)$ where $p$ is a polynomial with rational coefficients and $q$ is a vector with rational entries.
Since
\[
p(M^nq) = p(e_1^\top M^nq, \ldots, e_d^\top M^nq),
\]
each $u^{(k)}_n = e_k^\top M^nq$ is an LRS over $\rat$ (this can be seen, e.g., by applying the Cayley-Hamilton theorem), and LRS over $\rat$ are closed under addition and multiplication, the sequence $\seq{p(M^nq)}$ is itself an LRS over $\rat$.

The following is one of the most fundamental results about growth of linear recurrence sequences \cite[Theorem 2]{power-of-positivity}.
\begin{theorem}
	\label{thm:lrs-growth-s-units}
	Let $u_n = \sum_{i=1}^\ell p_i(n)\lambda_i^n$, where the right-hand side is in the exponential polynomial form.
	Further let $\rho = \max_i |\lambda_i|$.
	For every $\varepsilon > 0$, we have that
	\[
	u_n \ne 0 \Rightarrow |u_n| > (\rho - \varepsilon)^n
	\]
	for all sufficiently large $n$.
\end{theorem}

\paragraph*{Decision problems of linear recurrence sequences}

Sign patterns of LRSs have been studied for a long time. 
Two prominent open problems in this area are the \emph{Skolem Problem} and the \emph{Positivity Problem}. 
The Skolem Problem is to find an algorithm that, given an LRS $u_n$, decides if the set $Z = \{n\colon u_n=0\}$ is non-empty. 
The most well-known result in this direction is the celebrated Skolem-Mahler-Lech theorem, which (non-constructively) shows that $Z$ is a union of a finite set and finitely many arithmetic progressions.
In particular, it shows that a non-degenerate $\seq{u_n}$ can have only finitely many zeros.
The Positivity Problem, on the other hand, asks to find an algorithm that determines if $u_n \geq 0$ for all $n$. 
It is known to subsume the Skolem Problem as well as certain long-standing open problems in Diophantine approximation \cite[Chapter 2]{karimov-thesis}.

\subsection{Markov Chains} 
\label{sub:prelim_MC}
A finite-state \emph{discrete-time Markov chain} (DTMC) $M$ is a tuple $(S,  P,\iota )$, where $S$ is a finite set of states, $P: S \times S \rightarrow [0, 1]$ is the transition probability function satisfying $\sum_{s'\in S}P_{s s'} = 1$ for all $s \in S$ and $\iota : S\rightarrow [0, 1]$ is the initial distribution, such that $\sum_{s\in S}^{ } \iota_{init}(s) = 1$. For algorithmic problems, all transition probabilities are assumed to be rational.
A finite path $\rho$ in $M$ is a finite sequence $s_0s_1 \ldots s_n$ of states such that $P(s_i, s_{i+1}) > 0$ for all $0 \leq i \leq n-1$. 
We say that a state $t$ is reachable from $s$ if there is a finite path from $s$ to $t$. If all states are reachable from all other states, we say that $M$ is \emph{irreducible}; otherwise, we say it is \emph{reducible}. A set $B\subseteq S$ of states is called a bottom strongly connected component (BSCC)
if it is strongly connected, i.e., all states in $B$ are reachable form all other states in $B$ and if there are no outgoing transitions, i.e., $P(s,t)>0$ and $s\in B$ implies $t\in B$.

 Without loss of generality we identify $S$ with $\{1,\dots, d\}$ for $d=|S|$. Then, overloading notation, we consider $P\in\mathbb{R}^{d\times d}$ as a matrix with $P_{ij}=P(j,i)$ for $i,j\leq d$. 
 Likewise, we consider $\iota$ to be a (column)  
 vector in $\mathbb{R}^d$. 
 Then, the sequence of distributions over states after $k$ steps is given by $P^k \iota$, which forms a stochastic LDS.
 We also write $P^{(k)}_{ij}$ for $(P^k)_{ij}$, which is the probability to move from state $j$ to $i$ in exactly $k$ steps.
 Further, we say that the matrix $P$ is irreducible if the underlying Markov chain is irreducible.
  The \emph{period} $d_i$ of a state $i$ is 
    $d_i = \operatorname{gcd}\{m \geq 1 : P^{(m)}_{ii} > 0\}$.
    If $d_i = 1$, then we call the state $i$ \emph{aperiodic}. A Markov chain (and its transition matrix) is aperiodic if and only if all of its states are aperiodic.  
    The period of a Markov chain $M$ (and, equivalently, of its transition matrix $P$) is the least common multiple of the periods of the states of $M$.

    A vector $\pi\in \mathbb{R}^d$ is called a stationary distribution of the Markov chain  if
    a) $\pi$ is a distribution, i.e., $\pi_j \geq 0$ for all $j$ with $1\leq j \leq d$, and $\sum_{j =1}^d \pi_j = 1$, and
    b) $\pi$ is stationary, i.e., $\pi = P \pi$. 
    %
    For aperiodic Markov chains, it is known that the sequence $(P^k \iota)_{k\in \nat}$ of probability distributions over states  converges to a stationary distribution 
    $\pi$, which can be computed in polynomial time (see   \cite{Kulkarni1995,BK08}).
    
    \subsection{Kronecker's theorem}
    \label{sec:kronecker}
   	We write $\torus$ for the interval $[0,1)$ and $\{x\}$ for the fractional part of $x\in\mathbb{R}$, i.e., $\{x\}=x-\lfloor x \rfloor $.
    Let $\lambda_1,\ldots,\lambda_\ell \in \com$ with $\lambda_i = e^{\im 2\pi \theta_i}$ for all $i$, where $\theta_i \in \torus$.
    A \emph{multiplicative relation} of $\lambda_1,\ldots,\lambda_\ell$ is a tuple $(k_1,\ldots,k_\ell) \in \intg^\ell$ such that $\lambda_1^{k_1}\cdots\lambda_\ell^{k_\ell} = 1$.
    We write $G(\lambda_1,\ldots,\lambda_\ell)$ for the set of all multiplicative relations of $(\lambda_1,\ldots,\lambda_\ell)$, which is a free abelian group.
    If this group is consists only of the neutral element, we say that $\lambda_1,\ldots,\lambda_\ell$ are \emph{multiplicatively independent}.
    By a deep result of Masser~\cite{masser-mult-relations}, there exists a fixed polynomial $p$ such that $G$ has a basis $B$ of at most $\ell$ elements such that for each $v \in B$, $\Vert v\Vert _\infty < p(\Vert \lambda_1\Vert +\ldots+\Vert \lambda_\ell\Vert )^\ell$.
    Hence a basis of $G$ can be computed in polynomial space (given $\lambda_1,\ldots,\lambda_\ell$) by simply enumerating all possible bases satisfying Masser's bound.
    Note that $\lambda_1,\ldots,\lambda_\ell$ are multiplicatively independent if and only if $1, \theta_1,\ldots,\theta_\ell$ are linearly independent over $\rat$.
    
    We are now ready to state Kronecker's theorem in Diophantine approximation and its consequences for linear recurrence sequences.
    \begin{theorem}
    	Let $\theta_1,\ldots,\theta_\ell \in \torus$, 
    	$
    	\bm{\theta}_n = 
    	(\{n\theta_1\},\ldots,\{n\theta_\ell\})
    	$ for all $n$, and $G = \{(n_1,\ldots,n_\ell)\in\intg^\ell \st n_1\theta_1 + \cdots +n_\ell\theta_\ell \in \intg\}$.
    	The sequence $(\bm{\theta}_n)_{n\in\nat}$ is dense in
    	\[
    	\{(x_1,\ldots,x_\ell) \in \torus^\ell \mid k_1x_1+\cdots+k_\ell x_\ell \in \intg \: \textrm{ for all } (k_1,\ldots,k_\ell) \in G\}.
    	\]
    \end{theorem}
    \begin{corollary}
    	\label{thm:kronecker-cor1}
    	If $1,\theta_1,\ldots,\theta_\ell$ are linearly independent over $\rat$, then $(\bm{\theta}_n)_{n\in\nat}$ is dense in $\torus^\ell$.
    \end{corollary}
    \begin{corollary}[\unexpanded{See \cite[Prop.~3.5]{ouaknine2014ultimate-pos-for-simple-lrs}}]
    	\label{thm:kronecker-cor2}
    	Let $\gamma_1,\ldots,\gamma_\ell\in e^{\im 2\pi \torus}$.
    	The sequence $(\gamma_1^n,\ldots,\gamma_\ell^n)_{n\in\nat}$ is dense in
    	\[
    	\{(z_1,\ldots,z_\ell) \in (e^{\im 2\pi  \torus})^\ell \mid G(\gamma_1,\ldots,\gamma_\ell) \subseteq G(z_1,\ldots,z_\ell)\}.
    	\]
    \end{corollary}
    \begin{corollary}
    	\label{thm:kronecker-cor3}
    	Let $u_n = \sum_{i=1}^\ell c_i\lambda_i^n$ be an LRS that is not identically zero, where $\lambda_i, c_i \in \com$ and $|\lambda_i| = 1$ for all $i$.
    	There exists $c > 0$ such that $|u_n| > c$ for infinitely many $c$.
    \end{corollary}
    \begin{proof}
    	Let $k$ be such that $u_k \ne 0$.
    	By \autoref{thm:kronecker-cor2}, there exist infinitely many~$n$ such that $|(\lambda_1^n,\ldots,\lambda_\ell^n) - (\lambda_1^k,\ldots,\lambda_\ell^k)| < \varepsilon$.
    	The statement then follows from the continuity of the function $(z_1,\ldots,z_\ell) \mapsto \sum_{i=1}^\ell c_i z_i$. 
    \end{proof}
    
    \subsection{Logical theories and o-minimality}
    \label{sec:omin-prelims}
    A \emph{structure} $\mathbb{M}$ consists of a universe $U$, constants $c_1,\ldots,c_k \in U$, predicates $P_1,\ldots,P_l$ where each $P_i \subseteq U^{\mu(i)}$ for some $\mu(i) \ge 1$, and functions $f_1,\ldots,f_m$ where each $f_i$ has the type $f_i \st U^{\delta(i)} \to U$ for some $\delta(i) \ge 1$.
    By the \emph{language} of the structure $\mathbb{M}$, written  $\Lcal_{\Mb}$, we mean the set of all well-formed first-order formulas constructed from  symbols denoting  the constants $c_1,\ldots, c_k$, predicates $P_1,\ldots,P_l$, and functions $f_1,\ldots, f_m$, as well as variables and the symbols $\forall, \exists, \land, \lor, \lnot, =$.
    A \emph{theory} is simply a set of sentences, i.e.\ formulas without free variables.
    The theory of the structure $\mathbb{M}$, written $\operatorname{Th}(\Mb)$, is the set of all sentences in the language of $\mathbb{M}$ that are true in $\mathbb{M}$.
    A theory $\Tcal$ is \emph{decidable} if there exists an algorithm that takes a sentence $\varphi$ and decides whether $\varphi \in \Tcal$.
    Below is a list of structures relevant to this work.
    
    \begin{itemize}
    	\item[(a)] Let $\rel_0  = \langle \rel; 0,1,<,+,\cdot \rangle$, which is the ring of real numbers.
    	Observe that using the constants $0,1$ and the addition, we can obtain any constant $c \in \nat$.
    	Hence every atomic formula in $\Lcal_{or}$ with $k$ free variables is equivalent to $p(x_1,\ldots,x_k) \sim 0$, where $p$ is a polynomial with integer coefficients and $\sim$ is either $>$ or the equality.
    	By the Tarski-Seidenberg theorem, $\operatorname{Th}(\rel_0)$ admits quantifier elimination and is decidable.
    	\item[(b)] Let $\rexp = \langle \rel; 0,1,<,+,\cdot, \exp \rangle$, the real numbers augmented with the exponentiation function.
    	The theory of this structure is decidable assuming Schanuel's conjecture, a unifying conjecture in transcendental number theory.
    	\item[(c)] Finally, let  $\rebt$ denote the structure $\langle \rel; 0,1,<,+,\cdot, \exp, \widetilde{\cos}, \widetilde{\sin}\rangle$, which is $\rexp$ augmented with \emph{bounded trigonometric functions}: For $x \in [0,2\pi]$, $\widetilde{\cos}(x) = \cos(x)$ and $\widetilde{\sin}(x) = \sin(x)$. 
    	For all other values of $x$, we have $\widetilde{\cos}(x) = \widetilde{\sin}(x) = 0$.
    \end{itemize}
    
    We say that a structure $\Sb$ \emph{expands} $\Mb$ if $\Sb$ and $\Mb$ have the same universe and every constant, function, and relation of $\Mb$ is also present in~$\Sb$.
    We will only need structures expanding $\rel_0$.
    A set $X \subseteq U^d$ is \emph{definable} in a structure $\mathbb{M}$ if there exist $k \ge 0$, a formula $\varphi$ in the language of $\mathbb{M}$ with~$d+k$ free variables, and $a_1,\ldots,a_k \in U$ such that for all $x_1,\ldots,x_d \in U$, $\varphi(x_1,\ldots,x_d, a_1,\ldots,a_k)$ is true if and only if $(x_1,\ldots,x_d) \in X$.
    We say that $X$ is \emph{definable in $\Mb$ without parameters} if we can take $k = 0$ above.
    Similarly, a function is definable (without parameters) in $\Mb$ if its graph is definable (without parameters) in $\Mb$.
    
    A structure $\Mb$ expanding $\rel_0$ is \emph{o-minimal} if every set definable in $\Mb$ has finitely many connected components.
    The structures $\rel_0$, $\rexp$, and $\rebt$ are all o-minimal \cite{vdD1994bounded-analytic}.
    All o-minimal structures admit cell decomposition \cite[\S 4.2]{vdD-geometric-categories}.
    In particular, every function definable in an o-minimal structure is piecewise continuous and every subset of $[0,1]^m$ definable in an o-minimal structure is Jordan-measurable.
	 \section{Mean payoff}
	
	In this section we study the \emph{mean payoff} of an orbit, which is the average weight collected per step in the long run. For an LDS given by $M\in \mathbb{Q}^{d\times d}$ and $q\in\mathbb{Q}^d$, and a weight function $w\colon \mathbb{R}^d \to \mathbb{R}$, 
	we define the  mean payoff of the orbit as
	\[
	\MP_w(M,q) \coloneqq  \lim_{n\to \infty} \frac{1}{n} \sum_{k=0}^{n-1} w(M^k q).
	\]
	We show how to compute $\MP_w(M,q)$ in three cases: when $w$ is polynomial (\autoref{sec:polynomial_MP}), when the orbit is bounded and $w$ is continuous (\autoref{sec:bounded_MP}), and when $w$ is bounded and o-minimal (\autoref{sec:o-min-mean-payoff}).
	Let us now establish some of the key tools that we will need.
	
	Recall that we denote the interval $[0,1)$ by $\torus$ and the fractional part of $x$ by $\{x\}$.
	For $\mathbf{x} = (x_1,\ldots,x_m) \in \torus^m$, let 
	\[
	\sigma(\mathbf{x}) = (e^{\im 2\pi x_1},e^{-\im 2\pi x_1},\ldots, e^{\im 2\pi x_m},e^{-\im 2\pi x_m}).
	\]
	We will link mean payoff to certain integrals using Weyl's equidistribution theorem from ergodic theory.	\begin{theorem}[\unexpanded{\cite{Weyl1916}, see \cite{kelmendi2022} for an exposition}]
		\label{thm:weyl}
		Let $\theta_1,\ldots,\theta_m \in \torus$ be such that $1,\theta_1,\ldots,\theta_m$ are linearly independent over $\rat$, and $U$ be a Jordan measurable subset of $\torus^m$.
		Writing $\bm{\theta}_n = (\{n\theta_1\},\ldots,\{n\theta_m\})$, we have that
		\[
		\lim_{n\to\infty} \frac{|\{0\leq k < n \mid \bm{\theta}_k  \in U \}|  }{n} = \mathcal{L}(U).
		\]
		Consequently, for every (piecewise) continuous $g \st \torus^m \to \rel$, 
		\[
		\lim_{n\to\infty} \frac{1}{n}\sum_{k=0}^{n-1} g(\bm{\theta}_k) = \int_{\mathbf{x} \in \torus^m} g(\mathbf{x}) d\mathbf{x}
		\]
		where $d\mathbf{x} = dx_1\cdots dx_m$.
	\end{theorem}
	
	We will apply this theorem to orbits of linear dynamical systems in Sections \ref{sec:bounded_MP} and \ref{sec:o-min-mean-payoff} via the following lemma.
	\begin{lemma}
		\label{thm:subsequences-dense}
		Let $\lambda_1,\ldots,\lambda_\ell \in \alg$ where $\lambda_i = \rho_i \gamma_i$ with $\rho_i > 0$, $|\gamma_i| = 1$ and $\gamma_i = e^{\im2\pi\theta_i} \in \alg$.
		We can compute $R,m > 0$ for $1 \le i \le m$ such that $1, \theta_1,\ldots,\theta_m \in \torus$ are linearly independent over $\rat$, and polynomials $p_{i,r}$ for $1\le i \le \ell$ and $0 \le r < R$ such that
		\[
		\lambda_i^{nR+r} = \rho_i^{nR+r} \cdot p_{i,r}(\sigma(\bm{\theta}_n))
		\] 
		for all $i,n$, where $\bm{\theta}_n = (\{n\theta_1\},\ldots,\{n\theta_m\})$.
	\end{lemma}
	\begin{proof}
		The first step is to compute a basis of the group of multiplicative relations $G(\gamma_1,\ldots,\gamma_\ell)$ as described in \autoref{sec:kronecker}.
		Next, select a largest multiplicatively independent subset of $\{\gamma_1,\ldots,\gamma_\ell\}$.
		Without loss of generality we can take this to be $\{\gamma_1,\ldots,\gamma_m\}$ for some $m \le l$.
		Then for every $i > m$, we have that $\gamma_1,\ldots, \gamma_m, \gamma_i$ are multiplicatively independent, i.e.
		\[
		\gamma_1^{k_{i,1}}\cdots\gamma_m^{k_{i,m}} = \gamma_i^{k_i}
		\]
		for some positive integer $k_i$ and $k_{i,j} \in \intg$ for $1 \le j \le m$.
		Choose $R = k_{m+1} \cdots k_\ell$.
		Then for $1\le i \le m$, $\lambda_i^{nR+r} = \rho_i^{nR+r} (\gamma^r_i \cdot \gamma_i^{nR})$, and for $i > m$,
		\[
		\lambda_i^{nR+r} = \rho_i^{nR+r} \cdot \gamma_i^{nR+r} = \rho_i^{nR+r} \cdot \gamma_i^r \cdot \left(\gamma_1^{ \frac{R \cdot k_{i,1} } {k_i}} \cdots  \gamma_m^{\frac{R \cdot k_{i,m} }{ k_i}}\right)^n. 
		\]
		By construction of $R$, we have that $\frac{R \cdot k_{i,1} } {k_i}$ is an integer for all $i > m$.
	\end{proof}

\subsection{Polynomial weight-functions}
\label{sec:polynomial_MP}

In this section fix $M \in \rat^{d\times d}$, $q \in \rat^d$, and a polynomial $p \in \rat[X_1,\ldots,X_d]$.
We prove the following in this section.
\begin{theorem}
	\label{thm:mean-payoff-polynomial-weight}
	It is decidable whether the mean payoff $\MP_w(M,q)$ 
	exists, in which case it is rational and effectively computable.
\end{theorem}

In order to analyse the mean payoff of the orbit, first recall from \autoref{sec:lrs} that the sequence $(p(M^n q))_{n\in\nat}$ is an LRS over $\rat$.
The following lemma states that the sequence of partial sums of the weights is also a rational LRS. 

\begin{lemma}
\label{thm:LRS_poly_weight}
The sequence $u_n = \sum_{k=0}^n p(M^k q)$ is an LRS over $\rat$.
\end{lemma}
\begin{proof}
As discussed in \autoref{sec:lrs}, $w_n = p(M^n q)$ defines a rational LRS.
Suppose $\seq{w_n}$ satisfies a recurrence relation
\[
w_{n+d} = a_0w_{n} + \ldots + a_{d-1} w_{n+d-1}
\]
where $a_0,\ldots,a_{d-1} \in \rat$.
Then
\[
u_{n+d+1} = u_{n+d} + a_{d-1}(u_{n+d}-u_{n+d-1}) + \cdots + a_{0}(u_{n+1}-u_n)
\]
and hence $\seq{u_n}$ itself is a rational LRS of order at most $d+1$.
\end{proof}

Computing $\MP_w(M,q)$ therefore boils down to computing  $\lim_{n \to \infty} u_n/n$ for a rational LRS $(u_n)_{n\in\nat}$.

\begin{restatable}{lemma}{theoremfive}
\label{thm:MP_LRS}
	Given a rational LRS $\seq{u_n}$, it is decidable whether $\lim_{n \to \infty} \frac{u_n}{n}$ exists.
	When the limit exists, it is rational and effectively computable.
\end{restatable}
\begin{proof}
	Suppose $\seq{u_n}$ is not eventually zero and write
	$u_n = \sum_{i=1}^m p_i(n)\lambda_i^n$ where $m > 0$, $\lambda_1,\ldots,\lambda_m$ are pairwise distinct, $|\lambda_1| \ge \ldots \ge |\lambda_m| > 0$, and  each $p_i$ is not identically zero.
	By \autoref{thm:lrs-growth-s-units}, for $\epsilon>0$ we have  that $|u_n| > (|\lambda_1|-\epsilon)^n$ for infinitely many $n$.
	Hence if $|\lambda_1| > 1$, then the limit does not exist.
	Similarly, if $|\lambda_1| < 1$, then the limit is zero.
	Suppose therefore $|\lambda_1| = 1$.
	Let $k$ be the largest integer such that $|\lambda_i| = 1$ for all $i \le k$, and define $v_n = \sum_{i=1}^k p_i(n)\lambda_i^n$.
	It suffices to consider $\lim_{n \to \infty} v_n/n$ as $\lim_{n \to \infty} \sum_{i=k+1}^m p_i(n)\lambda_i^n = 0$.
	
	Write $v_n = \sum_{i=0}^l n^i \sum_{j=1}^{k_i} c_{i,j} \lambda_{i,j}^n$ where $\sum_{j=1}^{k_i} c_{i,j} \lambda_{i,j}^n$ is in the exponential polynomial form for all $i$.
	If $l = 0$, then $\seq{v_n}$ is bounded and $\lim_{n \to \infty} v_n/n = 0$.
	Assume therefore $l \ge 1$.
	Let 
	\[
	w_n = \sum_{j=1}^{k_l} c_{l,j} \lambda_{l,j}^n.
	\]
	By construction, $k_l \ge 1$, and by \autoref{thm:exp-poly-of-id-zero}, $\seq{w_n}$ is not identically zero.
	It is shown in \cite[Lemma 4]{braverman} that if $\lambda_{l,j} \ne 1$ for some~$j$ (which is implied by $k_l > 1$), then there exist $a,b \in \rel$ such that $a < b$, $w_n < a$ for infinitely many~$n$, and $w_n > b$ for infinitely many~$n$.
	Hence $\lim_{n \to \infty} v_n/n$ can exist only if $k_l = 1$ and $\lambda_{l,1} =1$.
	Suppose indeed that $k_l = 1$ and $\lambda_{l,1} =1$.
	Then $\lim_{n \to \infty}v_n/n$, when it exists, is equal to $\lim_{n \to \infty} \frac{n^lc_{l,1}}{n}$.
	Since $l \ge 1$, the latter limit exists if and only if $l = 1$, in which case it is equal to $c_{l,1}$ .
	
	To prove rationality of the limit,  suppose $\lim_{n \to \infty}v_n/n$ exists, in which case it is equal to $c_{l,1}$.
	There must exists $1 \le i \le k$ such that $\lambda_i = 1$ and $p_i(n)$ is equal to either $c_{l,1}$ or $nc_{l,1}$.
	Recall that $\alpha \in \com$ is rational if and only if it is fixed by every automorphism of $\com$.
	Since $\seq{u_n}$ takes rational values, $\sigma(u_n) = u_n$ for all $n \in \nat$ and $\sigma$ an automorphism of~$\com$.
	Moreover, $\sigma(u_n) = \sum_{i=1}^m\sigma(p_i(n))\sigma(\lambda_i)^n$ and
	$\sigma(\lambda_{l,1}) = \lambda_{l,1} = 1$ for every automorphism $\sigma$. 
	By the uniqueness of the exponential polynomial representation, $\sigma(c_{l,1}\lambda_{l,1}^n) = c_{l,1}\lambda_{l,1}^n$ for all $n$ and $\sigma$, which implies that $c_{l,1}$ is rational.
\end{proof}
Our main result (\autoref{thm:mean-payoff-omin-main}) now follows immediately from \autoref{thm:MP_LRS} and \autoref{thm:LRS_poly_weight}.
\qed

\paragraph*{Complexity of computing the mean payoff}
Recall that if the limit of $\seq{\frac{u_n} n}$ exists then it is equal to a coefficient of some $p_i$ appearing in the exponential polynomial solution of $\seq{u_n}$.
Hence the complexity of computing the limit is bounded by the complexity of computing the exponential polynomial.
If the description length of $\seq{u_n}$ is $\mathcal{I}$ and its order is~$d$, the time required to compute the exponential polynomial representation of $u_n$ is polynomial in $\mathcal{I}^d$ \cite[Chapter~2.2]{karimov-thesis}. 
In particular, computing the mean payoff requires polynomial time if we assume the dimension $d$ of the LDS to be fixed.

\subsection{Bounded orbits}
\label{sec:bounded_MP}

In this section, fix a continuous weight function $w \st \rel^d \to \rel$, as well as $M \in \rat^{d\times d}$ and $q \in \rat^d$ such that  $(M^nq)_{n \in  \nat}$ is bounded.
We will show that $\MP_w(M,q)$ has a simple integral representation.
Our approach can be explained geometrically as follows.
A bounded orbit of an LDS approaches a ``limit shape''--which is the set of accumulation points of the orbit-- closer and closer.
This allows us to express the mean payoff in terms of an integral of the weight function over this limit shape. 
This integral computes the ``average'' value of the weight function on the limiting shape.
Of course, we have to carefully ensure that we also know how ``frequently'' the orbit approaches different parts of the limiting shape.
Let us look at an example that illustrates our approach.

\begin{example}
	Let $w \colon \rel^3\to\rel$ be a continuous weight function and consider the LDS 
	\[
	M = \begin{bmatrix}
		3/5 & 4/5 & 0 \\
		-4/5 & 3/5 & 0 \\
		0 & 0 & 1/2
	\end{bmatrix}
	\quad \text{ and } \quad 
	q=\begin{bmatrix}
		1\\
		0\\
		1
	\end{bmatrix}.
	\]
	The first two coordinates evolve under a repeated application of a rotation.
	In the complex plane, this rotation is given by multiplication with $3/5-4/5 i$.
	As $3/5-4/5 i$ is not a root of unity, the orbit never reaches a point with $(1,0)$ in the first two coordinates again. In fact, the first two components of the orbit are dense in the unit circle in $\rel^2$.
	Furthermore, by, for example, Weyl's equidistribution theorem, each sub-interval of the unit circle is visited with frequency proportional to its length.
	The third component is halved at every step and converges to $0$. As the weight function is continuous, we can hence treat the third coordinate as equal to $0$ when determining the mean payoff.
	More formally, the set of accumulation points of the orbit is
	\[
	L=\{v = (v_1,v_2,v_3) \in\rel^3 \mid v_3=0, |v|=1\}
	\] 
	which can be parametrised via 
	$g(\theta) = (
		\cos(2\pi \theta),
		\sin(2\pi \theta),
		0)$, $g \st \torus \to \rel^3$.
	As this parametrisation moves through the circle with constant speed reflecting the fact that the orbit is ``equally distributed'' over the circle in the first two components, we can now express the mean payoff of the orbit with respect to the weight function $w$ as
	\[
	\MP_w(M,q) = \int_0^1 w (
	\cos(2\pi \theta),
	\sin(2\pi \theta),
	0) \,\,d\theta.
	\]
\end{example}
We now move on to the proof.
For $1 \le i \le d$, let $u^{(i)}_n = e_i M^n q$, where $e_i$ is the $i$th standard basis vector.
We will need the following lemma.

\begin{lemma}
	For all $1\le i \le d$,
	\[
	u^{(i)}_n = v^{(i)}_n + w^{(i)}_n
	\]
	where $(v^{(i)}_n)_{n\in\nat}$ is a simple LRS over $\ralg$ whose characteristic roots all lie on $\torus$, and $(w^{(i)}_n)_{n\in\nat}$ is an LRS over $\ralg$ whose characteristic roots all have magnitude less than one, i.e. $\lim_{n\to \infty} w^{(i)}_n = 0$.
\end{lemma}
\begin{proof}
	Fix $2 \le i \le d$, and write 
	\[
	u^{(i)}_n = \sum_{j=1}^\ell p_j(n)\lambda_j^n
	\]
	where the right-hand side is in the exponential polynomial form.
	If $|\lambda_j| >~1$ for some $j$, then by \autoref{thm:lrs-growth-s-units}, $|u^{(i)}_n|$ diverges to infinity, contradicting boundedness. 
	Therefore, $|\lambda_j| \le 1 $ for all $j$.
	Let $J = \{j \st \lambda_j \in \torus\}$ and $m = \max_{1 \le j \le \ell} \deg(p_j)$.
	We will show that $m = 0$.
	Write
	\[
	u^{(i)}_n = n^m \sum_{j \in J_1} c_j \lambda_j^n + \sum_{j=1}^\ell q_j(n)\lambda_j^n
	\]
	where $J_1 = \{j \st \deg(p_j)=m\}$ and $\deg(q_j) < m$ for all $j \in J$.
	By \autoref{thm:kronecker-cor3}, there exists $c > 0$ such that $|\sum_{j \in J_1} c_j \lambda_j^n | > c$ for infinitely many $n \in \nat$.
	It follows that if $m > 0$, then the first summand determines the sign of $u_n^{(i)}$ and hence $\liminf_{n\to \infty} u^{(i)}_n = \infty$, which again contradicts the boundedness assumption. 
\end{proof}

Let $v^{(i)}_n, w^{(i)}_n$  be as above.
We have that for all $N \in \nat$,
\[
\MP_w(M,q) =
\lim_{n\to \infty} \frac{1}{n} \sum_{k=N}^{N+n-1} w(M^k q).
\]
From the continuity of $w$ and the fact that $\lim_{n\to\infty} w_n^{(i)} = 0$ we deduce that
\[
\MP_w(M,q) =
\lim_{n\to \infty} \frac{1}{n} \sum_{k=0}^{n-1} w(v_k^{(1)}, \ldots, v_k^{(d)}).
\]
Next, let $\Lambda = \{\lambda_1,\ldots,\lambda_\ell\} \subset \alg$ be such that $|\lambda_i| = 1$ for all $i$ and the characteristic roots of every $(v_n^{(i)})_{n\in\nat}$ belong to $\Lambda$.
Invoking \autoref{thm:subsequences-dense} compute $\theta_1,\ldots,\theta_m \in \torus$ that are linearly independent over $\rat$, $R >0$, and the polynomials $p_{i,r}$ for $1 \le i \le \ell$ and $0 \le r < R$.
Write $\bm{\theta}_n$ for $(\{n\theta_1\}, \ldots, \{n\theta_m\})$ as usual.
We now prove the main result of this section.
\begin{theorem}
	\label{thm:mp_bounded}
	We can compute a function $f \st \rel^{m} \to \rel$ such that
	\[
	\MP_w(M,q) = \int_{\mathbf{x} \in \torus^m} f(\mathbf{x}) d \mathbf{x} 
	\]
	where $d\mathbf{x} = dx_1\cdots dx_m$.
\end{theorem}
\begin{proof}
	Write $v^{i,r}_n = v^{(i)}_{nR+r}$ for $1\le i \le d$ and $0 \le r < R$.
	We have that 
	\begin{align}
		\MP_w(M,q) &= \lim_{n\to \infty} \frac{1}{n} \sum_{k=0}^{n-1} w(v^{(1)}_k,\ldots,v^{(d)}_k) \\
		&= \frac 1 R \sum_{r=0}^{R-1} \lim_{n\to \infty} \frac{1}{n} \sum_{k=0}^{n-1} w(v^{1,r}_k, \ldots, v^{d,r}_k).
		\label{eq::omin-mean-payoff-1}
	\end{align}
	Applying \autoref{thm:subsequences-dense}, because $|\lambda_i| = 1$ for all $\lambda_i \in \Lambda$, we have that $\lambda_i^{nR+r} = p_{i,r}(\sigma(\bm{\theta}_n))$.
	Since each $(v^{1,r}_n)_{n\in\nat}$ is simple and $v^{1,r}_n$ is a $\alg$-linear combination of $\lambda_1^{nR+r},\ldots,\lambda_\ell^{nR+r}$, we have that 
	\begin{align*}
		w(v^{1,r}_k, \ldots, v^{d,r}_k) &= w(q_{1,r}\circ \sigma(\bm{\theta}_k),\ldots,q_{d,r} \circ \sigma (\bm{\theta}_k))
	\end{align*}
	where each $q_{i,r} \st \rel^{2m} \to \rel$ is a polynomial with algebraic coefficients.
	Thus for each $r$ there exists an entrywise polynomial function $q_r$ such that
	\[
	w(v^{1,r}_k, \ldots, v^{d,r}_k)  =  w \circ q_r \circ \sigma (\bm{\theta}_k)
	\] 
	for all $k$.
	Note that $q_r \circ \sigma$ and thus $w \circ q_r \circ \sigma$ are continuous. 
	Since the weight function $w$ is  continuous by assumption, we can apply \autoref{thm:weyl} to each summand of \eqref{eq::omin-mean-payoff-1} to obtain
	\begin{align*}
		\MP_w(M,q) &= \frac 1 R \sum_{r=0}^{R-1} \int_{\mathbf{x}\in \torus^m} w \circ q_r \circ \sigma (\mathbf{x}) d\mathbf{x}\\
		&=  \int_{\mathbf{x}\in \torus^m} \frac{w\circ q_0 \circ \sigma(\mathbf{x}) + \cdots + w \circ q_{R-1} \circ \sigma(\mathbf{x})}{R} d\mathbf{x}. \qedhere
	\end{align*}
\end{proof}

\paragraph{Approximation of the mean payoff}
Although the functions $q_0,\ldots,q_{R-1},\sigma$ are simple and explicit, there is no general way to evaluate the $m$-fold multiple integral computed above.
Nevertheless, the numerical approximation of integrals is an extensively studied area (see, e.g., \cite{davis2007methods}). 
The functions $q_r, \sigma$ are differentiable and we can bound the modulus of the gradient of $q_r \circ \sigma$ on $\torus^m$.
Consequently, if the function $w$ can be approximated and is well-behaved, e.g., if $w$  is Lipschitz continuous with known upper bound for its Lipschitz constant, also the integrals that we obtained can be approximated to arbitrary precision.
For more details on conditions under which the integral can be approximated to arbitrary precision, we refer the reader to
\cite{davis2007methods}.

\subsection{O-minimal weight functions}
\label{sec:o-min-mean-payoff}

In this section, fix $M\in\rat^{d\times d}$ and $q \in \rat^d$, as well as a bounded weight function $w \st \rel^d\to [-b,b]$ that is definable in some o-minimal expansion $\mathbb{M}$ of $\rebt$.
Let $\lambda_1,\ldots,\lambda_d$ be the eigenvalues of $M$.
Compute $R$, $\theta_1,\ldots,\theta_m$ as in \autoref{thm:subsequences-dense} and write $\bm{\theta}_n = (\{n\theta_1\},\ldots,\{n\theta_m\})$.
Our main result is the following, which shows that o-minimality implies good ergodic properties.

\begin{theorem}
	\label{thm:mean-payoff-omin-main}
	We can compute a function $f \st \torus^m \to \rel$ definable in $\mathbb{M}$ such that 
	\[
	\MP_w(M,q) = \int_{\mathbf{x}\in\torus^m} f(\mathbf{x}) d\mathbf{x}
	\]
	where $d\mathbf{x} = dx_1\cdots dx_m$.
\end{theorem}
We prove \autoref{thm:mean-payoff-omin-main} in the remainder of this section.
For $n \in \nat$ and $1 \le i \le r$, let $u^{(i)}_n = e_i^\top M^n q$ and $u_n^{i,r} = u^{(i)}_{nR+r}$.
Recall that
\begin{align*}
	\MP_w(M,q) &= \lim_{n\to \infty} \frac{1}{n} \sum_{k=0}^{n-1} w(u^{(1)}_k,\ldots,u^{(d)}_k) \\
	&= \frac 1 R \sum_{r=0}^{R-1} \lim_{n\to \infty} \frac{1}{n} \sum_{k=0}^{n-1} w(u^{1,r}_k, \ldots, u^{d,r}_k).
\end{align*}
Applying \autoref{thm:subsequences-dense}, for all $i,r$ we can compute a function $f_{i,r}$ that is definable in $\rebt$ such that $u^{i,r}_n =f_{i,r}(n, \{n\theta_1\},\ldots,\{n\theta_m\})$ for all $n$. 
Because $\Mb$ extends $\rebt$ by assumption, we can therefore compute a function $g_{r} \st \rel \times \torus^m \to [-b,b]$ definable in $\mathbb{M}$ such that
\[
w(u^{1,r}_n, \ldots, u^{d,r}_n) = g_{r}(n, \bm{\theta}_n).
\]
We will next show that the sequence of functions $(g_{r,n})_{n\in\nat}$, defined by
\[
g_{r,n}(x_1,\ldots,x_m) = g_r(n,x_1,\ldots,x_m)
\] 
must converge pointwise for all $n$.
\begin{lemma}
	Let $h_n \st \torus^m \to [-b,b]$ for $n \in \nat$ be such that for all $\mathbf{x} = (x_1,\ldots,x_m) \in\torus^m$ and $y \in [-b,b]$,
	\[
	h_n(\mathbf{x}) = y \Leftrightarrow  \varphi(n, x_1,\ldots,x_m,y)
	\]
	where $\varphi$ is a fixed formula in the language of $\M$.
	There exists a function $h \st \torus^m \to [-b,b]$ definable in $\Mb$, whose representation can be computed effectively, such that for every $\mathbf{x} \in \torus^m$,
	\[
	\lim_{n\to \infty} h_{n}(\mathbf{x}) = h(\mathbf{x}).
	\]
\end{lemma}
\begin{proof}
	Fix $\mathbf{x} \in \torus^m$, and let $z$ be an accumulation point of  $(h_n(\mathbf{x}))_{n\in\nat}$, whose existence is guaranteed by the Bolzano-Weierstraß theorem.
	Define the function $f\st \rel \to \rel_{\ge 0}$ by
	\[
	f(t) = d \Leftrightarrow y \in [-b,b] \st \varphi(t, x_1,\ldots,x_m,y) \:\land\: |y-z| = d
	\]
	which, at $t \in \nat$, measures the distance from $h_t(\mathbf{x})$ to $z$.
	Since it is definable in the o-minimal structure $\Mb$, it is ultimately monotonic \cite[\S 4.1]{vdD-geometric-categories}.
	By construction of $z$, $\liminf_{t \to \infty} f(t) = 0$.
	It follows that $\lim_{t\to \infty} f(t) = 0$.
	The function $h$ is therefore defined by 
	\begin{align*}
		h(x_1,&\ldots,x_m) = y \Leftrightarrow 
		\\
		&\forall \varepsilon > 0.\: \exists t.\: \forall t' > t.\: \exists y' \st \varphi(t,x_1,\ldots,x_m,y') \:\land\: |y'-y|<\varepsilon. \qedhere
	\end{align*}
\end{proof}
Applying the lemma above, for $0 \le r < R$, let $g_r \colon \torus^m \to [-b, b]$ be the pointwise limit of $(g_{r,n})_{n \in \nat}$, which is definable in $\Mb$ and hence piecewise continuous (see \autoref{sec:omin-prelims}).
We next show that when computing the mean payoff, we can work with the limit $g_r$ rather than the exact sequence $(g_{r,n})_{n \in \nat}$.
This bring us closer to applying the equidistribution theorem, which requires a fixed function that is applied at all to the orbit of the dynamical system given by a translation on $\torus^m$.
\begin{lemma}
	Let $0 \le r < R$. 
	We have that 
	\[
	\lim_{n\to \infty} \frac{1}{n} \sum_{k=0}^{n-1} g_{r,k}(\bm{\theta}_k) = \lim_{n\to \infty}  \frac{1}{n} \sum_{k=0}^{n-1} g_{r}(\bm{\theta}_k).
	\]
\end{lemma}
\begin{proof}
	For $\varepsilon > 0$ and $n \in \nat$, let $X_{\varepsilon, n}$ be the set of all $x \in \torus^m$ such that for all $m \ge n$, $|g_{r,m}(x)-g_r(x)| < \varepsilon$.
	We have that $X_{\varepsilon, n}$ is definable in $\Mb$ and thus (Jordan and Lebesgue) measurable (\autoref{sec:omin-prelims}).
	We denote the Jordan measure of a set $X \subseteq \torus^m$ by $\mathcal{C}(X)$.
	By construction of $g_r$ as the pointwise limit of $(g_{r,n})_{n\in\nat}$ and the definition of $X_{\varepsilon,n}$,
	\[
	\torus^m = \bigcup_{n \in \nat} X_{\varepsilon,n}
	\]
	for all $\varepsilon > 0$. 
	Moreover, $X_{\varepsilon,n} \subseteq X_{\varepsilon,m}$ for all $\varepsilon > 0$ and $n \le m$.
	By a standard property of measures (``continuity from below''),
	$\mathcal{C}(X_{\varepsilon, n})$ converges monotonically to 1 as $n \to \infty$ for all $\varepsilon > 0$.
	
	Let $\Delta > 0$.
	To prove our result it suffices to show that for all sufficiently large $n$, 
	\[
	\bigg|
	\frac{1}{n} \sum_{k=0}^{n-1} g_{r,k}(\bm{\theta}_k) - \frac{1}{n} \sum_{k=0}^{n-1} g_r(\bm{\theta}_k)
	\bigg| < \Delta
	\]
	which is equivalent to
	\[
	\bigg|
	\sum_{k=0}^{n-1} (g_{r,k}(\bm{\theta}_k) - g_r(\bm{\theta}_k)
	\bigg|
	< n\Delta.
	\]
	Recall that $[-b,b]$ is the image of $g_r$, and choose $\mu \in (0,1)$ and $\varepsilon, \delta > 0$ such that $(1-\mu+\delta) \cdot 2b < \frac \Delta 2$ and $(\mu + \delta) \cdot \varepsilon < \frac \Delta 4$.
	Let $N$ be such that $\mathcal{C}(X_{\varepsilon, N}) > \mu$, and write $X = X_{\varepsilon, N}$.
	For all $n \in \nat$, 
	\begin{multline*}
		\sum_{k=0}^{n-1} g_{r,k}(\bm{\theta}_k) - g_r(\bm{\theta}_k) =\\ \sum_{k=0}^{n-1} \mathbbm{1}(\bm{\theta}_k \in X) (g_{r,k}(\bm{\theta}_k) - g_r(\bm{\theta}_k)) + \sum_{k=0}^{n-1} \mathbbm{1}(\bm{\theta}_k \notin X) (g_{r,k}(\bm{\theta}_k) - g_r(\bm{\theta}_k))
	\end{multline*}
	where $\mathbbm{1}$ denotes the indicator function.
	By the Weyl equidistribution theorem, for all sufficiently large $n$ we have that
	\[
	\frac{1}{n} \sum_{k=0}^{n-1} \mathbbm{1}(\bm{\theta}_k\in X) \in (\mu-\delta,\mu+\delta)
	\]
	and 
	\[
	\frac{1}{n} \sum_{k=0}^{n-1} \mathbbm{1}(\bm{\theta}_k \notin X) \in (1-\mu-\delta,1-\mu+\delta).
	\]
	Hence for all sufficiently large $n$, 
	\[
	\bigg|
	\sum_{k=0}^{n-1} \mathbbm{1}(\bm{\theta}_k \notin X) (g_{r,k}(\bm{\theta}_k) - g_r(\bm{\theta}_k))
	\bigg|
	< (1-\mu+\delta)n \cdot 2b < \frac {n\Delta} {2}.
	\]
	Next, recall that for $k \ge N, |g_{r,k}(\bm{\theta}_k)-g_r(\bm{\theta}_k)| < \varepsilon$ by the construction of $X$.
	For $k < N$, we have that $|g_{r,k}(\bm{\theta}_k)-g_r(\bm{\theta}_k)| \leq 2b$.
	Hence for all sufficiently large $n$, 
	\[
	\bigg|
	\sum_{k=0}^{n-1} \mathbbm{1}(\bm{\theta}_k \in X) (g_{r,k}(\bm{\theta}_k) - g_r(\bm{\theta}_k))
	\bigg| <
	n(\mu+\delta) \cdot 2\varepsilon < \frac {n\Delta} {2}.
	\]
	It remains to apply the triangle inequality.
\end{proof}

Combining the lemma above with \autoref{thm:weyl} we deduce that
\[
\lim_{k\to \infty}  \frac{1}{n} \sum_{k=0}^{n-1} g_{r}(\bm{\theta}_k) = \int_{\mathbf{x} \in \torus^m} g_r(\mathbf{x}) d\mathbf{x}.
\]
Thus 
\begin{align*}
	\MP_w(M,q) &= \lim_{n\to \infty} \frac{1}{n} \sum_{k=0}^{n-1} w(u^{(1)}_k,\ldots,u^{(d)}_k) \\
	&= \frac 1 R \sum_{r=0}^{R-1} \lim_{n\to \infty} \frac{1}{n} \sum_{k=0}^{n-1} w(u^{1,r}_k, \ldots, u^{d,r}_k)\\
	&= \frac 1 R \sum_{r=0}^{R-1} \lim_{n\to \infty} \frac{1}{n} \sum_{k=0}^{n-1} g_r(\bm{\theta}_k)\\
	&= \int_{\mathbf{x} \in \torus^m} \frac{g_0(\mathbf{x})+\cdots+g_{R-1}(\mathbf{x})}{R} d\mathbf{x}.
\end{align*}
which concludes the proof of \autoref{thm:mean-payoff-omin-main}.

\paragraph*{Approximating mean payoff}
Recall that each $g_r \st \torus^m \to [-b,b]$ is definable in the o-minimal structure $\Mb$ and thus piecewise continuous (\autoref{sec:omin-prelims}).
Let us take $\Mb = \rebt$.
The first-order theory of this structure is decidable assuming Schanuel's conjecture.
Thus, by verifying truths of carious formulas (which can be done assuming Scahnuel's conjecture), we can compute an upper bound on the Lipschitz constant of each $g_r$ and approximate its value on $\mathbf{x} \in \torus^r$ with rational coordinates to arbitrary precision.
We can therefore, similarly to \autoref{sec:bounded_MP}, apply generic techniques to (conditionally) approximate mean payoff to arbitrary precision.

\subsection{Stochastic linear dynamical systems}
\label{sec:stochastic_MP}
Stochastic LDSs are a special case of LDSs with a bounded orbit. 
We will now show that for such systems, we can compute the mean payoff of the orbit under a continuous weight function by evaluating the weight function 
on finitely many points. 
In the aperiodic case, the orbit even converges to a single point, and consequently it suffices to evaluate the weight function once.
\begin{theorem}
	\label{thm:mean-payoff-aperiodic-markov}
Let $P\in \mathbb{Q}^{d\times d}$ be a stochastic, aperiodic matrix and $\iota\in \mathbb{Q}^d$ an initial distribution. 
Further let
$w\colon \mathbb{R}^d \to \mathbb{R}$ be a continuous weight function. Then,
$
\MP_w(P,\iota) = w(\pi)
$
where $\pi$ is the stationary distribution $\lim_{n\to \infty} P^n\iota$ of $P$, which is computable in polynomial time.
\end{theorem}

\begin{proof}
As described in \autoref{sub:prelim_MC}, we know that the orbit $(P^n\iota)_{n\in \nat}$ converges to a stationary distribution $\pi$ in this case, which can be computed in polynomial time \cite{Kulkarni1995,BK08}. 
As $w$ is continuous,
$\lim_{n\to \infty} w(P^n \iota) = w(\pi)$.
It follows that 
 \[
\MP_w(P,\iota) =  \lim_{n\to \infty} \frac{1}{n} \sum_{k=0}^{n-1} w(P^k \iota) =  w\big(\lim_{n\to \infty} P^n \iota\big) = w(\pi).\qedhere
 \]
 \end{proof}
Hence in the aperiodic case the computation of the mean payoff reduces to evaluating the function $w$ once on a rational point computable in polynomial time.
We next address the periodic case by splitting up the orbit into subsequences.

	For an irreducible and periodic Markov chain with period $L$, we have that $P^L$ is aperiodic and  $L\leq d$ by \cite[Theorem 1.8.4]{NorrisMC}.
	Together with \autoref{thm:mean-payoff-aperiodic-markov} this allows us to compute $\MP_w(P^L,P^r\iota)$, which satisfies
	\[
	\MP_w(P,\iota) = \frac{1}{L} \sum_{r=0}^{L-1} \MP_w(P^L,P^r\iota).
	\]
	That is, for an irreducible stochastic LDS we can divide $(P^{n}\iota)_{n\in\nat}$ into $L$ equally spaced subsequences and compute the mean payoff $\MP_w(P,\iota) $
	as the arithmetic mean of the mean payoffs of these subsequences.

\begin{theorem}
\label{thm:mp_stochastic_irreducible}
Let $P\in \mathbb{Q}^{d\times d}$ be a stochastic, irreducible matrix and $\iota\in\mathbb{Q}^d$ an initial distribution.
Let $w\colon \mathbb{R}^d \to \mathbb{R}$ be a continuous weight function.
 Then, we can compute points $\pi_0,\dots, \pi_{L-1}\in \mathbb{Q}^d$ in polynomial time for some $L\leq d$ such that 
$
\MP_w(P,\iota) = \frac{1}{L} \sum_{i=0}^{L-1} w(\pi_i)$.
\end{theorem}

Write $\Vert x \Vert$ for the bit length of $x$.
Since the points $\pi_0,\dots,\pi_{L-1}$ can be computed in polynomial time, they have bit length at most polynomial in the length of the original input.
Therefore, we have the following.
\begin{corollary}
\label{cor:approx_Markov}
Assume that the value $w(a)$ can be approximated in time $f_{w}(\Vert a\Vert , \epsilon) $ up to some precision $\epsilon\geq 0$ (where $\epsilon=0$ corresponds to exact computation) for all rational inputs $a\in \mathbb{Q}^d$.
There is a fixed polynomial $p$ such that the mean payoff $\MP_w(P,\iota)$ can be approximated up to precision~$\epsilon$ in  time  at most $d\cdot  f_{w}(p(\Vert  (P,\iota) \Vert , \epsilon) + p(\Vert  (P,\iota)\Vert )$.
 \end{corollary}

	When a Markov chain is reducible, 
	the states can be renamed in a way such that, the matrix representation of the Markov chain will contain distinct blocks corresponding to the  bottom strongly connected components (BSCCs) on the diagonal along with additional columns at the right representing states that do not belong to any BSCC:
    \[
    \begin{bmatrix}
    	\begin{matrix}
    		\resizebox{2em}{!}{\Huge $\square$}
    	\end{matrix}
    	& 0 ... 0 & 0 ... 0  & \ast &\ast \\
    	0...0
    	& \begin{matrix}
    		\resizebox{2.5em}{!}{\Huge $\square$}
    	\end{matrix}
    	& 0 ... 0  & \ast &\ast  \\
    	0 ... 0  & 0 ... 0  & \begin{matrix}
    		\resizebox{2em}{!}{\Huge $\square$}
    	\end{matrix} &  \ast &\ast  \\
0 ... 0 & 0...0&0....0& \ast &\ast 
    \end{bmatrix}
    \]
     Each block representing a BSCC constitutes an irreducible Markov chain. Assume we have $k$ blocks with periods $L_1, L_2,..., L_k$ correspondingly. Let $l$ be the least common multiple of the periods. Now we will have $l$ subsequences of the orbit each of which will converge. The convergence of the rows in the bottom is a result of the fact that Markov chain will enter a BSCC with  probability 1. 
     So, in general, we have $l$ subsequences of the orbit, all of which converge. 
We observe that  $l \leq d^d$, from which the following result follows:

\begin{theorem}
\label{thm:mp_stochastic_reducible}
Let $P\in \mathbb{Q}^{d\times d}$ be a stochastic matrix and $\iota\in\mathbb{Q}^d$ an initial distribution.
Let $w\colon \mathbb{R}^n \to \mathbb{R}$ be a continuous weight function.
Then, we can compute points $\pi_0,\dots, \pi_{l-1}\in \mathbb{Q}^d$ in exponential time for some $l\leq d^d$ such that 
$
\MP_w(P,\iota) = \frac{1}{l} \sum_{i=0}^{l-1} w(\pi_i)$.
\end{theorem}

As the transition matrix $P^l$ of the $l$ subsequences as well as the initial values $P^r \iota$ with $0\leq r< l$ can be computed in polynomial time by repeated squaring, each of the points $\pi_i$ with $0\leq i < l$ can be computed in polynomial time. 
Assuming that the value $w(a)$ can be approximated in time $f_{w}(\Vert a\Vert , \epsilon) $  for all rational inputs $a\in \mathbb{Q}^d$, we can hence conclude that there is again a fixed polynomial $q$ such that the mean payoff of reducible stochastic LDSs can be approximated 
 to precision $\epsilon$  in time bounded by 
$d^d \cdot  f_{w}(q(\Vert  (P,\iota) \Vert , \epsilon) + d^d \cdot q(\Vert  (P,\iota)\Vert )$ analogously to \autoref{cor:approx_Markov}.

	\section{Total (discounted) reward and satisfaction of energy constraints}
\label{sec:total}

In this section, we address the computation of total accumulated rewards and total discounted rewards as well as the problem to decide whether the 
accumulated reward ever drops below a given bound -- a problem known as the satisfaction of energy constraints -- for LDSs with polynomial weight functions.
First,
we prove that the total as well as the discounted accumulated weight of the orbit is  computable and rational if finite.
Afterwards, we discuss Baker's theorem and its consequences before applying these to solve the energy constraint problem for low-dimensional LDSs.
We furthermore provide two different hardness results showing that,  in dimension 4,
 the problem is hard with respect to certain open decision problems in Diophantine approximation and that restricting to stochastic LDSs and linear weight functions also does not lead to decidability in general.

\subsection{Total reward}
\label{sec:total_reward}
Let $M\in \mathbb{Q}^{d\times d}$ be a matrix, $q\in \mathbb{Q}^d$ be an initial vector, and $w\colon \mathbb{R}^d\to \mathbb{R}$ be a polynomial weight function with rational coefficients.
We define the \emph{total reward} as
\[
\mathrm{tr}(M,q,w) := \sum_{k=0}^\infty w(M^k q).
\] 
Likewise, for a discount factor $\delta\in (0,1) \cap \rat$ we define the \emph{total discounted reward} as 
\[
\mathrm{dr}(M,q,w,\delta) := \sum_{k=0}^\infty \delta^k \cdot w(M^k q).
\]
Both of these quantities, when they exist, can be determined effectively.
\begin{lemma}
	\label{thm:MP_LRS-corollary}
	For a rational LRS $\seq{u_n}$, it is decidable whether $\lim_{n \to \infty} u_n$ exists, in which case the limit is rational and effectively computable.
\end{lemma}
\begin{proof}
	The sequence $v_n = nu_n$ is also a rational LRS.
	It remains to observe that $\lim_{n \to \infty} u_n = \lim_{n \to \infty} v_n/n$ and apply \autoref{thm:MP_LRS}.
\end{proof}
\begin{theorem}
	\label{thm:total_discounted}
	It is decidable whether  $\sum_{k=0}^\infty w(M^kq)$ and $ \sum_{k=0}^\infty \delta^k \cdot w(M^k q)$ converge, in which case their values are rational and effectively computable.
\end{theorem}
\begin{proof}
	Let $u_n = \sum_{k=0}^n w(M^kq)$.
	As argued in \autoref{sec:polynomial_MP}, $\seq{u_n}$ is a rational LRS, and we can apply \autoref{thm:MP_LRS-corollary}.
	Similarly, let $v_n = \sum_{k=0}^\infty \delta^k \cdot w(M^k q)$.
	As $\seq{\delta^n}$ is itself a (rational) LRS and such LRS are closed under pointwise multiplication, $v_n$ is also a rational LRS.
	We again apply \autoref{thm:MP_LRS-corollary}.
\end{proof}

\subsection{Baker's theorem and its applications}
\label{sec:baker}
We now discuss Baker's theorem, which is the most important tool for quantitatively analysing growth of low-order linear recurrence sequences.
We will use it in the analysis of energy constraints in the next section.
A \emph{linear form in logarithms} is an expression of the form
\[
\Lambda = b_1 \Log \alpha_1 + \cdots + b_m \Log \alpha_m
\] where $b_i \in \intg$ and $\alpha_i \in \alg$ for all $1 \le i \le m$.
Here $\Log$ denotes the principal branch of the complex logarithm.
The celebrated theorem of Baker places a lower bound on $|\Lambda|$ in case $\Lambda \ne 0$.
Baker's theorem, as well as its $p$-adic analogue, play a critical role in the proof of \cite{tijdeman_distan_between_terms_algeb_recur_sequen} that the Skolem Problem is decidable for linear recurrence sequences of order at most 4, as well as decidability of the Positivity Problem for low-order sequences \cite{ouaknine13_posit_probl_low_order_linear_recur_sequen}.

\begin{theorem}[Special case of the main theorem in \cite{baker-rational-sharp-version-1993}]
	\label{baker-standard-baker}
	Let $\Lambda$ be as above, $D$ be the degree of the field extension $\rat(\alpha_1, \ldots, \alpha_m)/\rat$, and suppose $A,B \ge 3$ are such that $A > H(\alpha_i)$ and $B > |b_i|$ for all $1 \le i \le m$.
	If $\Lambda \ne 0$, then
	\[
	\log |\Lambda| > -(16mD)^{2(m+2)} (\log A)^m \log B.
	\] 
\end{theorem}
A consequence of Baker's theorem is the following \cite[Corollary 8]{ouaknine_simple-positivity}.
\begin{lemma}
	\label{baker-conseq}
		Let $\alpha,\beta\in \alg$ with $|\alpha| =1$.
	For all $n \ge 2$, if $\alpha^n \ne \beta$ then 
	$
	|\alpha^n -\beta| > n^{-C}
	$
	where $C$ is an effective constant that depends only $\alpha$ and $\beta$.
\end{lemma}
If $\alpha$ is not a root of unity, $\alpha^n = \beta$ holds for at most one $n$ which can be effectively bounded.
\begin{lemma}
	Let $\alpha, \beta \in \alg$ be non-zero, and suppose $\alpha$ is not a root of unity.
	There exists effectively computable $N\in \nat$ such that $\alpha^n \ne \beta$ for all $n > N$. 
\end{lemma}
\begin{proof}
	The \emph{Weil height} \cite[Chapter 3.2]{waldschmidt2013diophantine} of $\alpha$, denoted by $h(\alpha)$, is non-zero under the assumption on $\alpha$ and satisfies $h(\alpha^n) = nh(\alpha)$.
	We can therefore choose $N = \lceil h(\beta)/h(\alpha)\rceil$.
\end{proof}
Combining the two lemmas above, we obtain the following.
\begin{theorem}
	\label{baker-conseq-2}
	Let $\alpha, \beta \in \alg$, and suppose that $|\alpha| = 1$ and $\alpha$ is not a root of unity.
	There exist effectively computable $N, C \in \nat$ such that for all $n > N$, $|\alpha^n-\beta|>n^{-C}$. 
\end{theorem}
The next lemma summarises the family of linear recurrence sequences to which we can apply Baker's theorem.
\begin{lemma}
	\label{3d-lrs-lemma}
	Let $\gamma \in \{z \in \alg \st |z| = 1\}$ be not a root of unity, $r_1,\ldots,r_\ell \in~\rel$ be non-zero, and
	$
	u_n = \sum_{i=1}^m c_i\Lambda_i^n
	$
	be an LRS over $\rel$ where $m \ge 1$, $c_i, \Lambda_i \in \alg$ are non-zero for all~$i$, and $\Lambda_1,\ldots, \Lambda_m$ are pairwise distinct.
	Suppose each $\Lambda_i$ is in the multiplicative group generated by $\{\gamma, r_1, \ldots, r_\ell\}$.
	\begin{itemize}
		\item[(a)] There exists effectively computable $N_1$ such that $u_n \ne 0$ for all $n > N_1$.
		\item[(b)] For $n > N_1$, $|u_n| > L^n n^{-C}$, where $L = \max_{i} |\Lambda_i|$ and $C$ is an effectively computable constant.
		\item[(c)] It is decidable whether $u_n \ge 0$ for all $n$.
	\end{itemize}
\end{lemma}
\begin{proof}
	Define $\mathcal{D} = \{i \colon |\Lambda_i| = L\}$ and $\mathcal{R} = \{i \colon |\Lambda_i| < L\}$.
	The terms $c_i\Lambda_i^n$ with $i \in \mathcal{D}$ are called \emph{dominant}.
	We have
	\[
	u_n = \underbrace{\sum_{i \in \mathcal{D}} c_i \Lambda_i^n}_{v_n} +  \underbrace{\sum_{i \in \mathcal{R}} c_i \Lambda_i^n}_{z_n}. 
	\]
	We first investigate $|v_n|$ as $n \to \infty$.
	Recall that each $\Lambda_i$ is of the form $\gamma^{m_0}r_1^{m_1}\cdots r_\ell^{m_\ell}$, where $m_0, \ldots,m_\ell \in \intg$.
	That is, for all $i$, $\Lambda_i = |\Lambda_i|\gamma^{k_i}$ for some $k_i \in \intg$.
	Hence we can write
	\begin{equation}
		\label{baker-eq1}
		v_n = L^n \sum_{i=-K}^K b_i \gamma^{in}
	\end{equation}
	where each $b_i$ is equal to some $c_j$ and $b_K \ne 0$.
	The values of~$i$ range over $\{-K,\ldots,K\}$ because $v_n$ is real-valued and hence the summands in \eqref{baker-eq1} appear in conjugate pairs.
	We have 
	\[
	v_n = \gamma^{-Kn}L^n  \sum_{i=0}^{2K} b_{-K+i}\gamma^{in} =  \gamma^{-Kn}L^n  \prod_{i=0}^{2K}(\gamma^n-\alpha_i)
	\]
	where $\alpha_0, \ldots, \alpha_{2K} \in \alg$ are the zeros of the polynomial $p(z) = \sum_{i=0}^{2K} b_{-K+i}z^i$.
	Since $\gamma$ is not a root of unity, we can apply \autoref{baker-conseq-2} to each factor $(\gamma^n-\alpha_i)$ to conclude that there exist effectively computable $N_1,C$ such that $|v_n| >L^n n^{-C}$ for al $n > N_1$.
	Since $|\Lambda_i| < L$ for all $i \in \mathcal{R}$, there exists (effectively computable) $N_2$ such that $|v_n| > |z_n|$ for $n > N_2$.
	We have proven (a) and (b).
	
	Since $u_n$ is real-valued, by \autoref{thm:real-lrs-galois-condition} for each $1 \le i \le m$ there exists $1 \le j \le m$ such that $c_j\Lambda_j = \overline{c_i\Lambda_i}$.
	Hence both $v_n$ and $z_n$ are real-valued.
	By the analysis above $\operatorname{sign}(u_n) = \operatorname{sign}(v_n)$ for $n > N_2$.
	Hence to check if $u_n$ is positive we have to check whether $u_n \ge 0$ for $0 \le n \le N_2$ and $v_n \ge 0$ for $n > N_2$.
	We show how to do the latter.
	Since $\gamma$ is not a root of unity, $\seq{\gamma^n}$ is dense in $e^{\im2\pi \torus}$.
	Define $f(z) = z^{-K}p(z)$, noting that $v_n = L^n \cdot f(\gamma^n)$.
	Consider $Z \coloneqq f(e^{\im2\pi\torus}) \subset \rel$, which is compact and equal to the closure of $\{\gamma^{-Kn} L^n p(\gamma^n) \mid n \in \nat\}$.
	If $Z$ contains a negative number, then by density of $\seq{\gamma^n}$ in $e^{\im2\pi\torus}$, $v_n$ is negative for infinitely many~$n$.
	Hence $u_n < 0$ for infinitely many $n$.
	Otherwise, $v_n \ge 0$ for all $n$ and hence $u_n \ge 0$ for all $n > N_2$.
	This concludes the proof of~(c).
\end{proof}

\subsection{Satisfaction of energy constraints}
\label{sec:energy}
We next discuss \emph{energy constraints}.
We say that a series of real weights $(w_i)_{i\in \mathbb{N}}$ satisfies the energy constraint with budget $B$ if 
\[
\sum_{i=0}^k w_i \geq - B
\]
for all $k\in \mathbb{N}$.
We will prove that for LDS $(M,q)$ of dimension at most 3, satisfaction of energy constraints is decidable.
The proof is based on the fact that three-dimensional systems are tractable thanks to Baker's theorem \cite{karimov-ltl-3d}.
For higher-dimensional systems,  no such tractability result is known.
We will show that deciding satisfaction of energy constraints is, in general, at least as hard as the Positivity Problem, already with linear weight functions.
Before giving our decidability result, we need one final ingredient about partial sums of LRS.
Let $w_n = n^l \lambda^n$ for some $l \ge 0$ and $\lambda \in \alg$, and $u_n = \sum_{k=0}^n w_k$.
If $\lambda=1$, then $u_n = p(n)$, where $p$ is a polynomial of degree $l+1$ with rational coefficients.
If $\lambda \ne 1$, then $u_n = q(n)\lambda^n$, where $q$ is a polynomial of degree at most $l$ with algebraic coefficients satisfying $ q(n+1)\lambda = q(n)+ n^l$.
It follows that if the LRS $\seq{w_n}$ has only real eigenvalues, then so does the sequence given by $u_n = \sum_{k=0}^nw_k$.
Similarly, if $\seq{w_n}$ is diagonalisable and does not have~1 as an eigenvalue, then the same applies to $\seq{u_n}$.
In fact, the eigenvalues of $\seq{u_n}$ form a subset of the eigenvalues of $\seq{w_n}$.
\begin{theorem}
	\label{thm:energy_dimension_3}
	Let $M\in \mathbb{Q}^{3\times 3}$, $q\in\mathbb{Q}^3$, $\delta \le 1$ be a discount factor, and $w\colon \rel^3\to\rel$ be a polynomial weight function with rational coefficients.
	For $B\in \mathbb{Q}_{\geq 0}$, it 
	is decidable whether the weights $\seq{\delta^n \cdot w(M^nq)}$ satisfy the energy constraint with budget $B$.
\end{theorem}
\begin{proof}
	Let  $w_n = \delta^n \cdot w(M^nq)$ and $u_n = B + \sum_{i=0}^n w(M^iq)$.
	We have to decide whether $u_n \ge 0$ for all~$n$.
	First suppose $M$ has only real eigenvalues.
	Then $w_n$ and $u_n$ are both LRSs with only real eigenvalues.
	By taking subsequences if necessary, we can assume $\seq{u_n}$ is non-degenerate.
	Write
	$
	u_n = \sum_{i=1}^m p_i(n)\rho_i^n
	$
	where the right-hand side is in the exponential-polynomial form.
	In particular, for all~$i$, $p_i$ is not the zero polynomial.
	Since $\seq{u_n}$ is non-degenerate, without loss of generality we can assume $\rho_1 > \ldots > \rho_m> 0$.
	If $p_1(n)$ is negative for sufficiently large $n$, then the energy constraint is not satisfied.
	Otherwise, we can compute $N$ such that for all $n > N$, $u_n > 0$.
	It remains to check whether $u_n \ge 0$ for $0 \le n \le N$.
	
	Next, suppose $M$ has non-real eigenvalues $\lambda, \overline{\lambda}$, and a real eigenvalue $\rho$.
	Write $\gamma = \lambda/ |\lambda|$ and $r = |\lambda|$.
	Then 
	$u_n$ is of the form
	\[
	u_n = cn + \sum_{i=1}^m c_i \Lambda_i^n \coloneqq cn + v_n
	\]
	where $\Lambda_1, \ldots, \Lambda_m$ are pairwise distinct and in the multiplicative group generated by $r, \rho, \delta, \gamma$.
	Without loss of generality we can assume $c_i \ne 0$ for all $i$, but $c$ may be zero.
	If $\gamma$ is a root of unity of order $k>0$ (i.e. $\gamma^k=1$), then we can take subsequences $\seq{u^{(0)}_n}, \ldots, \seq{u^{(k-1)}_n}$, where $u^{(j)}_n = u_{nk+j}$ for $n \in \nat$ and $0 \le j < k$, and each $\seq{u^{(j)}_n}$ has only real eigenvalues.
	We can then apply the analysis above.
	Hereafter we assume $\gamma$ is not a root of unity.
	
	Suppose $c = 0$.
	Then \autoref{3d-lrs-lemma}~(c) applies and we can decide whether $u_n \ge 0$for all $n$.
	Next, suppose $c \ne 0$ and $L \coloneqq \max_i |\Lambda_i|$ satisfies $L \le 1$.
	We can compute $N_2$ such that $|cn| > |v_n|$ for all $n > N_2$.
	Hence in this case $u_n \ge 0$ for all $n$ if and only if $c > 0$ and $u_n > 0$ for $0 \le n \le N_2$.
	Finally, suppose $c \ne 0$ and $L > 1$.
	Applying \autoref{3d-lrs-lemma}~(b), there exists effectively computable $N_3$ such that $|u_n| > |cn|$ for $n > N_3$.
	Hence $u_n \ge 0$ for all $n$ if and only if $u_n \ge 0$ for $0 \le n \le N_3$ and $ \sum_{i=1}^mc_i\Lambda_i^{n} \ge 0$ for $n > N_3$.
	The latter can be decided by applying \autoref{3d-lrs-lemma}~(c) to the sequence $v_n =  \sum_{i=1}^m c_i \Lambda_i^{N_3+n} =  \sum_{i=1}^m (c_i\Lambda_i^{N_3})\Lambda_i^n$.
\end{proof}

\subsection{Positivity and Diophantine hardness}
\label{sec:hardness}
Recall that the energy satisfaction problem is to decide, given a matrix $M \in \rat^{d\times d}$, $q \in \rat^d$, $B \in \rat$, and a polynomial $p$ with rational coefficients, whether there exists $n$ such that $\sum_{k=0}^n p(M^kq) < B$.
This problem is at least as hard as the Positivity Problem already for stochastic linear dynamical systems and linear weight functions.
\begin{restatable}{theorem}{lemmatwentytwo}
	\label{thm:energy_Diophantine}
	The Positivity Problem can be reduced to the energy satisfaction problem restricted to a Markov chain $(M,q)$ and a linear function~$w$.
\end{restatable}
\begin{proof}
	It is known from \cite{AAOW2014,mihir} that the Positivity Problem for arbitrary LRS over $\rat$ can be reduced to the following problem: given a Markov chain $(M,q)$, decide whether there exists $n$ such that $e_1^\top M^nq \ge 1/2$, where $e_1 = (1,0,\ldots,0) \in \rel^d$.
	We reduce the latter to the energy satisfaction problem.
	Given a Markov chain $(M,q) \in \rat^{d\times d} \times \rat^d$, let
	\[
	P = \begin{bmatrix}
		M & \zerovec \\
		\zerovec & M
	\end{bmatrix}
	\]
	and $t =(\frac 1 2 q, \frac 1 2 Mq) \in \rat^{2d}$.
	Observe that $(P,t)$ is also a Markov chain.
	Moreover, $P^nt = (\frac 1 2 M^nq, \frac 1 2 M^{n+1}q)$ for all $n$.
	We choose the weight function $w(x_1,\ldots,x_{2d}) = 2(x_{d+1}-x_1)$ and $B = \frac 1 2 -e_1\cdot q$.
	Then
	\[
	w(P^nq) = e_1^\top M^{n+1}q - e_1^\top M^nq
	\]
	 and $u_n \coloneqq \sum_{k=0}^n w(P^nq) \ge B$ if and only if $e_1^\top M^{n+1}q \ge \frac 1 2$.
	Hence there does not exist $n$ such that $e_1^\top M^nq \ge \frac 1 2$ if and only if $e_1\cdot q < \frac 1 2$ and there does not exist $n$ such that $u_n < B$.
\end{proof}

The Positivity Problem for rational linear recurrence sequences of order 6 was shown in \cite{ouaknine13_posit_probl_low_order_linear_recur_sequen} to be \emph{Diophantine-hard}.
Specifically, for $r \in \rat$ and $\lambda \in \rat(\im)$ (that is, $\lambda = a + b\im$ where $a,b \in \rat$) let
\[
u^{\lambda,r}_n = -n + \frac{n}{2}(\lambda^n+\overline{\lambda^n}) + \frac{r\im}{2}(\overline{\lambda^n}-\lambda^n) = r\Ima(\lambda^n) - n\Rea(\lambda^n) + n.
\]
If for all $r \in \rat$ and $\lambda \in \rat(\im)$ we can decide whether $u^{\lambda,r}_n  \ge 0$ for all $n \ge 0$, then we could compute the \emph{Lagrange constants} of a large class of numbers, which would amount to a major mathematical breakthrough in number theory.
The following theorem states that a solution to the energy satisfaction problem for rational LDS in dimension 4 with polynomial weight functions would also yield the same breakthrough.
\begin{theorem}
\label{thm:diophantine4}
The  energy satisfaction problem for rational linear dynamical systems  in dimension 4 and polynomial weight functions with rational coefficients is Diophantine-hard. 
\end{theorem}

\begin{proof}
We prove the following:
If we can decide the energy satisfaction problem with $d = 4$, then we can decide for every $r, \lambda$ whether $u^{\lambda,r}_n \ge 0$ for all $n$.
Fix $r \in \rat$ and $\lambda = a+ bi \in\rat(i)$.
Define
\[
M = \begin{bmatrix}
	a&-b&1&0\\
	b&a&0&1\\
	0&0&a&-b\\
	0&0&b&a
\end{bmatrix}
\]
and the initial point $q = (0,0,0,1)$.
We have
\[
M^nq = (-n\Ima(\lambda^{n-1}), n\Rea(\lambda^{n-1}), -\Ima(\lambda^n), \Rea(\lambda^n)).
\]
Recall that for all $n\in\nat$, (i) $\Rea(\lambda^{n+1}) = a\Rea(\lambda^n)-b\Ima(\lambda^n)$ and (ii)  $\Ima(\lambda^{n+1}) = a\Ima(\lambda^n) + b\Rea(\lambda^n)$.
Hence there exist polynomials $p_1,p_2,p_3,p_4$ with rational coefficients such that  $p_1(M^nq) = n\Ima(\lambda^n)$, $p_2(M^nq) = n\Rea(\lambda^n)$, $p_3(M^nq) = \Rea(\lambda^n)$ and $p_4(M^nq) = \Ima(\lambda^n)$ for all $n$.
Next, consider
\[
w_n = u^{\lambda,r}_{n+1}-u^{\lambda,r}_{n} = r\Ima(\lambda^{n+1}) - n\Rea(\lambda^{n+1}) -   r\Ima(\lambda^{n}) + n\Rea(\lambda^{n}) + 1.
\]
Since $u^{\lambda,r}_0 = 0$, we have $u^{\lambda,r}_{n+1} = \sum_{k=0}^n w_n$.
Moreover, using facts (i), (ii) and the polynomials $p_1,\ldots,p_4$ we can construct a polynomial $p$ with rational coefficients such that $w_n = p(M^nq)$ for all $n$.
Hence $u^{\lambda,r}_n \ge 0$ for all $n$ if and only if the weights $\seq{p(M^nq)}$ satisfy the energy constraint with budget $B = 0$.
\end{proof}
	\section{Conclusion}
We have shown how to compute (or approximate) the mean-payoff and the total or discounted weight of the orbit of a rational linear dynamical system
for various combinations of classes of systems as well as  weight functions (see \autoref{tbl:overview}).
Remarkably, the results concerning  \emph{infinite horizon} questions (e.g. mean payoff, as opposed to satisfaction of energy constraints) do not rely on restrictions of the dimension in a stark contrast to decidability results 
for the Skolem \cite{tijdeman_distan_between_terms_algeb_recur_sequen,Verescchagin} and the Positivity \cite{ouaknine13_posit_probl_low_order_linear_recur_sequen,ouaknine_simple-positivity} problems, which themselves are special cases of reachability questions about the orbit of an LDS.
This is in line with vast decidability of various \emph{robust} versions of the reachability (alternatively, safety) problem of linear dynamical systems, which similarly does not depend on the dimension and is also proven using o-minimality \cite{karimov2024verification}.
Our results in \autoref{sec:o-min-mean-payoff} can be seen as a counterpoint to these decidability results regarding robust reachability: we show that o-minimality implies strong ergodic properties. 

For the question of whether an orbit of a rational LDS with a polynomial weight function satisfies an energy constraints, on the other hand, we have shown decidability for dimension 3 by utilising  deep number-theoretic tool that form the cornerstone of the theory of linear recurrence sequences, most notable Baker's theorem on linear forms in logarithms.
As it is expected with all finite (but not a priori bounded) horizon problems of linear dynamical systems, we obtain and Diophantine-hardness for the full problem, specifically already for systems in the ambient space $\rel^d$.
In fact, restricting the dynamical system to be stochastic (i.e.\ a Markov chain) and the weight function to be linear still results in a problem that is at least as hard as the Positivity Problem for linear recurrence sequences.
This, again, is unsurprising in light of \cite{AAOW2014,mihir} results that show that for the Skolem and Positivity problems, restricting the LDS to be a Markov chain does not change much.

In the future, this work can be extended in at least two ways.
Firstly, one can consider \emph{continuous-time} linear dynamical systems equipped with a weight function.
For such systems only hyperplane and halfspace reachability problems (which are analogues of the Skolem and Positivity problems, respectively) are well-understood, and the number-theoretic tools used in their analyses differ significantly from the ones used in this work \cite{chonev2023zeros}.
However, continuous-time linear dynamical systems also benefit from applications of o-minimality \cite{karimov2024verification}, and we believe that they also have good ergodic properties with respect to ``tame'' (e.g., o-minimal) weight functions.
Secondly, one can partition $\rel^d$ into a collection  of semialgebraic sets $S_1, \ldots, S_m$  and assign a fixed reward $w(S_i)$ to each $S_i$.
Here the reward received at time $n$ is $\sum_{i=1}^m\mathbbm{1}(M^nq \in S_i)w(S_i)$.
In this setting, already for $M$ that is a rotation in $\rel^2$, comparing the (discounted) total reward as well as the mean payoff against a given threshold appears to have deep connections to Diophantine approximation \cite{luca2022transcendence}.


	\bibliographystyle{elsarticle-num} 
	\bibliography{refs}
	
	
		
		
		
	
\end{document}